%% file: main_v_final.tex
\documentclass[a4paper,english, 11pt]{amsart}
\makeatletter
 \theoremstyle{plain}
\newtheorem{thm}{Theorem}[section]
\theoremstyle{plain}
  \newtheorem{prop}[thm]{Proposition}
\theoremstyle{plain}

\theoremstyle{plain}
 \newtheorem{lemma}[thm]{Lemma}
\theoremstyle{plain}

\theoremstyle{plain}

\theoremstyle{definition}
  
 \theoremstyle{definition}
  
  \newtheorem{exam}[thm]{Example}
  
\theoremstyle{remark}
\newtheorem{rmk}[thm]{Remark}
\numberwithin{equation}{section}
 
\usepackage{amsmath}
\usepackage{amssymb}
\usepackage{amscd}
\usepackage{amsthm}
\usepackage{array}
\usepackage{xcolor}
\usepackage[arrow,matrix,tips,all,cmtip]{xy}

\usepackage{stmaryrd}
\usepackage{hyperref}

\pdfpageheight\paperheight
\pdfpagewidth\paperwidth
\newcommand{\Z}{\mathbb{Z}}

\newcommand{\Q}{\mathbb{Q}}

\newcommand{\Qp}{\mathbb{Q}_p}
\newcommand{\R}{\mathbb{R}}

\newcommand{\F}{\mathbb{F}}

\newcommand{\fS}{\mathfrak{S}}

\newcommand{\fm}{\mathfrak{m}}

\newcommand{\fc}{\mathfrak{c}}

\newcommand{\cD}{\mathcal{D}}

\newcommand{\cH}{\mathcal{H}}

\newcommand{\cM}{\mathcal{M}}
\newcommand{\cN}{\mathcal{N}}
\newcommand{\cO}{\mathcal{O}}

\newcommand{\phz}{\varphi}

\newcommand{\Gal}{\mathrm{Gal}}
\newcommand{\Hom}{\mathrm{Hom}}

\newcommand{\GL}{\mathrm{GL}}

\DeclareMathOperator{\Mod}{Mod}
\DeclareMathOperator{\Mat}{Mat}

\DeclareMathOperator{\Fil}{Fil}
\DeclareMathOperator{\MF}{MF}
\DeclareMathOperator{\ev}{ev}

\DeclareMathOperator{\st}{{st}}

\newcommand{\Rep}{\operatorname{Rep}}
\newcommand{\crys}{\operatorname{crys}}

\newcommand{\wt}{\widetilde}
\newcommand{\D}{{\mathcal D}}
\newcommand{\M}{{\mathfrak M}}
\newenvironment{smallpmatrix}
  {\left(\begin{smallmatrix}}
  {\end{smallmatrix}\right)}
  
  \usepackage{mathrsfs}
  \usepackage{multicol}

\makeatother

\usepackage{enumitem}

\title[Reductions of semi-stable representations]{Reductions of 2-dimensional semi-stable representations with large $\mathcal L$-invariant}
\author{John Bergdall}
\address{Bryn Mawr College,
Department of Mathematics,
101 North Merion Avenue,
Bryn Mawr, PA 19010, USA}
\email{jbergdall@brynmawr.edu}

\author{Brandon Levin}
\address{Department of Mathematics,
University of Arizona, 
617 N Santa Rita Avenue, 
Tucson, Arizona 85721, USA}
\email{bwlevin@math.arizona.edu}

\author{Tong Liu}
\address{Department of Mathematics,
Purdue University, 
	 150 N. University Street, 
	West Lafayette, Indiana 47907, USA}
\email{tongliu@math.purdue.edu}
\date{\today}


\begin{document}

\begin{abstract}
We determine reductions of 2-dimensional, irreducible, semi-stable, and non-crystalline representations of $\Gal(\overline{\Q}_p/\Qp)$ with Hodge--Tate weights $0 < k-1$ and with $\mathcal L$-invariant whose $p$-adic norm is sufficiently large, depending on $k$. Our main result provides the first systematic examples of the reductions for $k \geq p$.
\end{abstract}
\maketitle

\setcounter{tocdepth}{1}
\tableofcontents

\input{introduction-final}
%\newpage
\input{theoretical-final}
%\newpage
\input{explicit-final}
%\newpage
\input{descent-final}
%\newpage

\bibliography{semistable_bibliography}
\bibliographystyle{abbrv}

\end{document}

%% file: introduction-final.tex
\section{Introduction}
Let $p$ be a prime number and $\overline{\mathbb Q}_p$ be an algebraic closure of the $p$-adic numbers $\mathbb Q_p$. The goal of this article is to determine the reductions of certain 2-dimensional $p$-adic representations of $G_{\mathbb Q_p} = \Gal(\overline{\mathbb Q}_p/\mathbb Q_p)$ that are semi-stable and not crystalline in the sense of Fontaine (\cite{Fontaine-RepresentationSemiStable}). Examples of such representations arise from local $p$-adic representations associated with eigenforms with $\Gamma_0(p)$-level.

\subsection{Main result}
Write $v_p$ for the $p$-adic valuation on $\overline \Q_p$, normalized so that $v_p(p) = 1$. Choose $\varpi \in \overline \Q_p$ such that $\varpi^2 = p$. Then, for each integer $k\geq 2$ and each $\mathcal L \in \overline \Q_p$, there is a 2-dimensional filtered $(\varphi,N)$-module $D_{k,\mathcal L} =  \overline \Q_pe_1 \oplus \overline \Q_pe_2$ where, in the basis $(e_1,e_2)$, we have:
\begin{align}\label{eqn:phi-module-basis}
\varphi &= \begin{pmatrix} \varpi^k & 0 \\ 0 & \varpi^{k-2} \end{pmatrix} 
&
N &= \begin{pmatrix} 0 & 0 \\ 1 & 0 \end{pmatrix}
&
\Fil^i D_{k,\mathcal L} &= \begin{cases} D_{k,\mathcal L} & \text{if $i \leq 0$;}\\
\overline \Q_p\cdot(e_1 + \mathcal Le_2) & \text{if $1 \leq i \leq k-1$;}\\
\{0\} & \text{if $k \leq i$}.
\end{cases}
\end{align}
Each $D_{k,\mathcal L}$ is weakly-admissible, so a theorem of Colmez and Fontaine implies there is a unique 2-dimensional $\overline \Q_p$-linear representation $V_{k,\mathcal L}$ of $G_{\Q_p}$ such that $D_{k,\mathcal L} = D_{\st}^{\ast}(V_{k,\mathcal L})$. Up to a  twist by a crystalline character, the representations $V_{k,\mathcal L}$ enumerate all $\overline \Q_p$-linear 2-dimensional semi-stable and non-crystalline representations of $G_{\Q_p}$. They are irreducible except if $k = 2$.

We aim to determine the semi-simple mod $p$ reductions $\overline V_{k,\mathcal L}$ of $V_{k,\mathcal L}$.  Twenty years ago, Breuil and M\'ezard determined $\overline V_{k,\mathcal L}$ for even $k < p$ and any $\mathcal L$ (\cite[Th\'eor\`eme 4.2.4.7]{BreuilMezard-Multiplicities}). Guerberoff and Park recently studied odd $k < p$ (\cite[Theorem 5.0.5]{GuerberoffPark-Semistable}). The reader who takes a moment to examine the cited theorems should be left with an impression of the complicated dependence of $\overline V_{k,\mathcal L}$ on $\mathcal L$, and that is just for $k < p$.

Prior results are limited by their ambition to determine $\overline V_{k,\mathcal L}$ for all $\mathcal L$. Here, we focus on determining $\overline V_{k,\mathcal L}$ for any $k$ while restricting to $\mathcal L$ that place $V_{k,\mathcal L}$ in a $p$-adic neighborhood of a crystalline representation (see Section \ref{subsec:overview}). Write $\mathbb Q_{p^2}$ for the unramified quadratic extension of $\mathbb Q_p$, $\chi$ for its quadratic character modulo $p$, and $\omega_2$ for a niveau 2 fundamental character on $G_{\mathbb Q_{p^2}}$.

\begin{thm}[Theorem \ref{thm:main-thm-in-text}]\label{thm:thm-intro}
Assume $ k \geq 4$ and $p \neq 2$. Then, if 
\begin{equation*}
\displaystyle v_p(\mathcal L) < 2 - \frac{k}{2} - v_p((k-2)!),
\end{equation*}
 then $\overline V_{k,\mathcal L} \cong \operatorname{Ind}_{G_{\mathbb Q_{p^2}}}^{G_{\mathbb Q_p}}(\omega_2^{k-1} \chi)$.
\end{thm}
To be accurate, our method proves Theorem \ref{thm:thm-intro} when $k \geq 5$ or $p=3$ and $k=4$. The theorem holds for $k =4$ and $p \geq 5$ by the work of Breuil--M\'ezard, and it is consistent with their work and the work of Guerberoff--Park for $5 \leq k < p$. Our method also directly obtains a result for $k = 3$ and $k = 4$ with a weaker bound. See Remark \ref{remark:final-remark} for a more detailed discussion. Our exclusion of $p=2$ is more fundamental (see Remark \ref{rmk:p=2-intro}).

\begin{rmk} When $k < p$ and $k$ is even, the bound in Theorem \ref{thm:thm-intro} is optimal by the results of Breuil--M\'ezard. The same can be said if $5 \leq k < p$ and $k$ is odd, by the work of Guerberoff--Park.    We do not know to what extent the bound is optimal for higher weights (see Section \ref{subsec:globalcontext}).
\end{rmk}

Theorem \ref{thm:thm-intro} is a natural analog of widely-studied theorems that determine reductions of 2-dimensional, irreducible, crystalline representations of $G_{\mathbb Q_p}$. For instance, Buzzard and Gee (\cite{BuzzardGee-SmallSlope}) developed a strategy to determine reductions of certain crystalline representations, with unbounded Hodge--Tate weights, using the $p$-adic local Langlands correspondence. We do not know whether a direct analog for semi-stable, but non-crystalline, representations has been tried or, even, if such an approach would be feasible.

Another approach in the crystalline case is via integral $p$-adic Hodge theory. Berger, Li, and Zhu and Berger proved local constancy results for reductions of crystalline representations using Wach modules (\cite{BergerLiZhu-SmallSlopes,Berger-LocalConstancy}). Recently, the first two named authors of this article improved the Berger--Li--Zhu result using Kisin modules (\cite{BergdallLevin-BLZ}). Those are what we will use here, also. One incentive to write the prior article was as training to conduct the current research.

Finally, an indirect approach to calculating $\overline V_{k,\mathcal L}$ is explained in a recent preprint of Chitrao, Ghate, and Yasuda (\cite{ChitraoGhateYasuda-Semistable}), though their investigation heads in a  interesting separate direction from ours.

\subsection{Overview of strategy}\label{subsec:overview}
We now describe our strategy, first re-contextualizing Theorem \ref{thm:thm-intro} through the lens of local constancy of reductions as in \cite{BergerLiZhu-SmallSlopes,Berger-LocalConstancy,BergdallLevin-BLZ}.

The parametrization of semi-stable and non-crystalline representations by $\mathcal L\in \overline{\mathbb Q}_p$ extends to a $\mathbb P^1(\overline{\mathbb Q}_p)$-parametrization with a crystalline representation at $\infty$. Namely, for $\mathcal L \neq 0$ we consider $D_{k,\mathcal L}$ with basis $(e_1',e_2') = (e_1,\mathcal L e_2)$ in which case, rather than \eqref{eqn:phi-module-basis}, we have
\begin{align}\label{eqn:phi-module-basis-2}
\varphi &= \begin{pmatrix} \varpi^k & 0 \\ 0 & \varpi^{k-2} \end{pmatrix} 
&
N &= \begin{pmatrix} 0 & 0 \\ \mathcal L^{-1} & 0 \end{pmatrix}
&
\Fil^i D_{k,\mathcal L} &= \begin{cases} D_{k,\mathcal L} & \text{if $i \leq 0$;}\\
\overline \Q_p\cdot(e_1' + e_2') & \text{if $1 \leq i \leq k-1$;}\\
\{0\} & \text{if $k \leq i$}.
\end{cases}
\end{align}
Thus, $D_{k,\mathcal L} \rightarrow D_{k,\infty}$ as $\mathcal L^{-1} \rightarrow 0$, where $D_{k,\infty}$ is the filtered $(\varphi,N)$-module with the same $\varphi$ and filtration as \eqref{eqn:phi-module-basis-2} but with $N = 0$. In fact, $D_{k,\infty} \cong D_{\crys}^{\ast}(V_{k,\infty})$ where $V_{k,\infty}$ is a 2-dimensional {\em crystalline} representation of $G_{\mathbb Q_p}$ whose Frobenius trace is $a_p = \varpi^{k-2} + \varpi^{k}$.  Replacing the filtered $(\varphi,N)$-modules with Galois representations, we have $V_{k,\mathcal L} \rightarrow V_{k,\infty}$ as $\mathcal L^{-1} \rightarrow 0$ (see the description in \cite[Section 4.5-4.6]{Colmez-Trianguline} in terms of the space of trianguline representations, for instance). Thus, $\overline V_{k,\mathcal L} \cong \overline V_{k,\infty}$ for $\mathcal L^{-1} \rightarrow 0$. Furthermore, $v_p(a_p) = \frac{k-2}{2}$ and so $\lfloor{\frac{k-1}{p}}\rfloor < v_p(a_p)$, except if $p = 2$ or $k$ is small, and so $\overline V_{k,\infty} \cong \operatorname{Ind}_{G_{\mathbb Q_{p^2}}}^{G_{\mathbb Q_p}}(\omega_2^{k-1} \chi)$ by \cite[Corollary 5.2.3]{BergdallLevin-BLZ}. We have reduced the theorem to the question:\ at which point as $\mathcal L^{-1} \rightarrow 0$, do we have $\overline V_{k,\mathcal L} \cong \overline V_{k,\infty}$?

We recall the relationship between reductions and Kisin modules, now. To ease notations, assume for the remainder of this subsection that $k$ is even and $\mathcal L \in \mathbb Q_p$, so $V_{k,\mathcal L}$ and $V_{k,\infty}$ are defined over $\mathbb Q_p$. Let $\mathfrak S = \mathbb Z_p[\![u]\!]$, and write $\varphi: \mathfrak S \rightarrow \mathfrak S$ for the Frobenius map $\varphi(u) = u^p$. Then, consider the category $\Mod_{\mathfrak S}^{\varphi,\leq k-1}$ of $\varphi$-modules over $\mathfrak S$ with height $\leq k-1$ (\cite{Kisin-FCrystals}). Objects in this category, which are called Kisin modules, are finite free $\mathfrak S$-modules $\mathfrak M$ equipped with a $\varphi$-semilinear operator $\varphi: \mathfrak M \rightarrow \mathfrak M$ such that the cokernel of the linearization $\varphi^{\ast}\mathfrak M \rightarrow \mathfrak M$ is annihilated by $E(u)^{k-1}$, where $E(u) = u+p$. When $\mathfrak M$ satisfies the {\em monodromy condition}, Kisin's theory constructs a canonical semi-stable representation $V_{\mathfrak M}$ such that $D_{\st}^{\ast}(V_{\mathfrak M}) \cong \mathfrak M/u\mathfrak M[1/p]$, for a certain filtration and monodromy on the right-hand side. Furthermore, $\overline V_{\mathfrak M}$ is determined by $\mathfrak M/p\mathfrak M[u^{-1}]$ as a $\varphi$-module over $\mathbb F_p(\!(u)\!)$. The challenge in calculating $\overline V_{\mathfrak M}$ this way is determining $\mathfrak M$ from $V_{\mathfrak M}$ or, equivalently, $D_{\st}^{\ast}(V_{\mathfrak M})$. That task was carried out for $V_{k,\infty}$ in \cite[Theorem 5.2.1]{BergdallLevin-BLZ}. 

The heart of this article is a two-step argument to do the same for $V_{k,\mathcal L}$ as $\mathcal L^{-1} \rightarrow 0$. The difficulty presented by non-trivial monodromy on $D_{k,\mathcal L}$  requires us to develop a new technique to pass from filtered $(\varphi,N)$-modules to Kisin modules. We make use of a category intermediate between filtered $(\varphi,N)$-modules and Kisin modules. Namely, write $\Mod_{S_{\mathbb Q_p}}^{\varphi,\leq k-1}$ for the category of $\varphi$-modules over $S_{\mathbb Q_p} = \mathbb Z_p[\![u,\frac{E^p}{p}]\!][\frac 1 p]$ with height $\leq k-1$. This category is close to certain filtered $(\varphi,N)$-modules considered by Breuil (\cite{Breuil-Griffiths}). Adapting Breuil's work, we explicitly construct a canonical object $\mathcal M_{k,\mathcal L} \in \Mod_{S_{\mathbb Q_p}}^{\varphi,\leq k-1}$ such that  if $\mathfrak M \in \Mod_{\mathfrak S}^{\varphi,\leq k-1}$ and $\mathcal M_{k,\mathcal L} \cong \mathfrak M \otimes_{\mathfrak S} S_{\mathbb Q_p}$, then $V_{\mathfrak M} \cong V_{k,\mathcal L}$. Explicit means, for any (non-zero) $\mathcal L$, we determine a basis of $\mathcal M_{k,\mathcal L}$ and an exact formula for $\varphi$ in that basis. This is  where we overcome the difficulty of non-trivial monodromy on $D_{k,\mathcal L}$.

The second step is to descend $\mathcal M_{k,\mathcal L}$ from $S_{\mathbb Q_p}$ to $\mathfrak S$ when $\mathcal L^{-1} \rightarrow 0$, thus producing an $\mathfrak M$ for $V_{k,\mathcal L}$. Here, we view $S_{\mathbb Q_p}$ as subring of $R_2$, where $R_2$ is the ring of $p$-adic rigid analytic functions on $|u| \leq p^{-1/2}$ (using that $p\neq 2)$. Section 4 of \cite{BergdallLevin-BLZ} presents a row reduction algorithm for semilinear operators that, under certain conditions, can descend from $R_2$ to $\mathfrak S$.   Specifically, the main theorem in {\em loc.\ cit.}\ gives a sufficient condition to descend $\mathcal M_{k,\mathcal L} \otimes_{S_{\mathbb Q_p}} R_2$ to $\mathfrak S$. Saving the details for later, we use the explicit calculation of $\mathcal M_{k,\mathcal L}$ to check those conditions are met when $v_p(\mathcal L) < 2 - {k\over 2} + v_p((k-2)!)$.

\begin{rmk}\label{rmk-thedifference} As already discussed, our approach in the first step is more general than \cite{BergdallLevin-BLZ} as it applies in the semi-stable, non-crystalline case.  In fact, the method is quite general and can be used (with a suitable descent process) to compute reductions for higher dimensional semi-stable representations.   For example, the third author has used the strategy here to compute reductions of irreducible 3-dimensional crystalline representations of $G_{\Q_p}$ with Hodge-Tate weights $\{0 , r , s\}$ satisfying  $2 \leq  r  \leq p-2$ and $ p+2 \leq  s \leq r+ p-2$. See \cite{Liu-3dimensional}.
\end{rmk}

\begin{rmk}\label{rmk:p=2-intro}
We exclude $p=2$ twice. The second time, when we embed $S_{\mathbb Q_p}$ into $R_2$ is likely technical. However, we also exclude $p=2$ when referencing the calculation of $\overline V_{k,\infty}$ in \cite{BergdallLevin-BLZ}, and that seems crucial:\ our strategy is based not just on knowing $\overline V_{k,\infty}$, but also how to construct a Kisin module for $V_{k,\infty}$. Including $p=2$, here would necessarily require calculating $\overline V_{k,\infty}$ when $p=2$ as well. We note the formula $\overline V_{k,\infty}\cong \operatorname{Ind}_{G_{\mathbb Q_{p^2}}}^{G_{\mathbb Q_p}}(\omega_2^{k-1}\chi)$ should still be true, but we cannot justify it.
\end{rmk}

\subsection{Global context} \label{subsec:globalcontext}
We end this introduction with a discussion of the global situation. Suppose $N \geq 1$ and $f = \sum a_n(f)q^n$ is a cuspidal (normalized) eigenform of level $\Gamma_1(N)$, weight $k \geq 2$, and nebentype character $\psi_f$. Eichler--Shimura and Deligne famously associated with $f$ a 2-dimensional, irreducible, continuous representation $V_f$ of $\Gal(\overline \Q/\mathbb Q)$. We normalize $V_f$ so that for $\ell \nmid Np$ the restriction $V_f|_{D_\ell}$ to $D_\ell$, a decomposition group at $\ell$, is unramified and the characteristic polynomial of a geometric Frobenius element is $t^2 - a_\ell(f)t + \psi_f(\ell)\ell^{k-1}$. The representation $V_f|_{D_p}$ is semi-stable when $p^2 \nmid N$ and the conductor of $\psi_f$ is prime-to-$p$; it is crystalline when $p \nmid N$ (\cite{Saito-padicHodgeTheory}).

We assume now that $V_f|_{D_p}$ is semi-stable and non-crystalline, in which case we define the $\mathcal L$-invariant of $f$ to be the unique $\mathcal L_f \in \overline{\mathbb Q}_p$ such that $V_f|_{D_p} \cong V_{k,\mathcal L_f}$. The $\mathcal L$-invariant defined this way is called the Fontaine--Mazur $\mathcal L$-invariant (it agrees with \cite[Section 12]{Mazur-Monodromy} up to a sign). It is a local quantity, but it famously arises in global situations. Examining how it arises allows us to provide global examples where Theorem \ref{thm:thm-intro} applies and to connect $\mathcal L$-invariants to global phenomena on $p$-adic families.

Theorem \ref{thm:thm-intro} determines $(\overline{V_f}|_{D_p})^{\mathrm{ss}}$ in arbitrary weights $k \geq p$ as long as $v_p(\mathcal L_f)$ is sufficiently negative, but it is not immediately obvious that eigenforms exist with $v_p(\mathcal L_f)$ so negative. Recent research, however, sheds light on the situation. For instance, Gr\"af (\cite{Graef-Linvariant}) and Anni, B\"ockle, Gr\"af, and Troya (see \cite{ABGT-Linvariants}, which builds on \cite{ColemanStevensTeitelbaum-Experiments}) have developed the theory and practice needed to calculate the multiset of valuations of $\mathcal L$-invariants in a fixed weight and level. Pollack has also developed computer code to calculate $\mathcal L$-invariants. His method, which dates to the early 2000's, uses the appearance of $\mathcal L$-invariants in exceptional zero phenomena for $p$-adic $L$-functions. That method is being written up as part of a joint investigation by Pollack and the first author.

Using their works, both Pollack and Gr\"af kindly calculated some $\mathcal L$-invariants for us. In Table \ref{table:L-inv}, we partially list the $p$-adic valuations found when $p = 3$ and $N = 51 = 3 \cdot 17$.
\begin{table}[htp]
\renewcommand{\arraystretch}{1.1}
\caption{$3$-adic valuations of some $\mathcal L$-invariants.}
\begin{center}
\begin{tabular}{|c|l|}
\hline
$k$ & $v_3(\mathcal L_f)$ for newforms $f \in S_k(\Gamma_0(51))$\\
\hline
4 & $-2,-1,0,0,\dotsc$\\ 
6 & $-3,-2,-1,-1,-1,\dotsc$\\
8 & $-3,-3, -\frac 3 2, -\frac 3 2, -\frac 3 2, -\frac 3 2, -1, \dotsc$\\
\hline
\end{tabular}
\end{center}
\label{table:L-inv}
\end{table}% 
Note, the bound in Theorem \ref{thm:thm-intro} is $v_3(\mathcal L_f) < 0$ in weight $k = 4$ and $v_3(\mathcal L_f) < -2$ in weight $k=6$, so Table \ref{table:L-inv} provides two examples of  Theorem \ref{thm:thm-intro} in weight $k = 4$ and one example of Theorem \ref{thm:thm-intro} in weight $k = 6$, though none in weight $k=8$.

Let's look further at $p = 3$ and $k = 6$ and the boundary case $v_3(\mathcal L) = -2$ in Theorem \ref{thm:thm-intro}. Pollack's code, in fact, reports not just $v_3(\mathcal L_f)$ for each newform $f$, but also $\overline V_f$. This refined data shows that the eigenforms with weight $k =6$ and $v_3(\mathcal L_f)$ equal to $-3$ and $-2$ have isomorphic global Galois representations modulo $3$. Since Theorem \ref{thm:thm-intro} applies to $v_3(\mathcal L) = -3$, we see that there exists $\mathcal L$-invariants with $v_3(\mathcal L) = -2$ for which the conclusion of Theorem \ref{thm:thm-intro} continues to hold. More numerical data is required before theorizing about the sharpness of the bound in Theorem \ref{thm:thm-intro}.

The $\mathcal L$-invariants also arise, globally, from $p$-adic families. Namely, $f$ lives in a $p$-adic family of eigenforms parametrized by weights $k \in \mathbb Z_p$ and $\mathcal L_f = -2 \operatorname{dlog} a_p(k) = -2{a_p'(k)\over a_p(f)}$ (\cite[Corollarie 0.7]{Colmez-L_invariants}).  This appearance reveals an obstruction to the ``radius'' of the largest ``constant slope'' family through $f$. Indeed, for $p\neq 2$, \cite[Theorem 4.3]{Bergdall-ConstantSlopes} implies $v_p(\mathcal L_f^{-1}) \leq m(f)$ where $m(f)$ is the least positive integer such that $f$ lives in a $p$-adic family of eigenforms $f'$ with $v_p(a_p(f')) = v_p(a_p(f))$ and weight $k' \equiv k \bmod (p-1)p^{m(f)}$. 

So, ruling out exceptions to Theorem \ref{thm:thm-intro}, we have $v_p(\mathcal L_f) < 2 - {k\over 2} - v_p((k-2)!)$ implies
\begin{itemize}
\item $(\overline V_f|_{D_p})^{\mathrm{ss}} \cong \operatorname{Ind}_{G_{\mathbb Q_{p^2}}}^{G_{\mathbb Q_p}}(\omega_2^{k-1} \chi)$, and
\item $m(f) > {k\over 2} - 2 + v_p((k-2)!) \approx {k-2\over 2}  + {k\over p-1}$.
\end{itemize}
To connect these, if $k\not\equiv 1 \bmod p+1$, then  $\overline V_f|_{D_p}$ is irreducible. On the other hand, condition (2) generically implies $m(f) > {k-2\over 2} = v_p(a_p(f))$. The fact that $m(f) > v_p(a_p(f))$ occurs in a situation where $\overline V_f|_{D_p}$ is irreducible is not a coincidence. It follows a pattern of counter-examples to a conjecture of Gouv\^ea and Mazur, which is related to the $m(f)$, found by Buzzard and Calegari (\cite{BuzzardCalegari-GouveaMazur}). See \cite[Section 9]{Bergdall-ConstantSlopes} for more discussion.

\subsection{Acknowledgements}
We owe the heuristic reframing in Section \ref{subsec:overview} to comments by Laurent Berger  and Christophe Breuil during the conference ``G\'eom\'etrie arithm\'etique, th\'eorie des repr\'esentations et applications'' at the Centre International de Rencontres Math\'ematiques (CIRM) in Luminy, France. Part of this collaboration also took place during the workshop ``Moduli spaces and modularity'' at Casa Matem\'atica Oaxaca (CMO). We thank both Berger and Breuil for their comments and both CIRM and CMO for their hospitality. 

Finally, acknowledgments are due for the discussion in Section \ref{subsec:globalcontext}. First, the data reported in Table \ref{table:L-inv} in the first preprint version of this article was inaccurate. Because of that, we drew faulty conclusions, which have now been removed, on the strength of the bound in Theorem \ref{thm:thm-intro}. We thank Robert Pollack for calculating the original data and then alerting us to the error. We also especially thank Peter Gr\"af for replicating the newly reported data, using his alternative method.

J.B.\ was partially supported by an NSF Grant (DMS-1402005) and a Simons Collaboration Grant (\#713782). B.L.\ was supported by a grant from the Simons Foundation/SFARI (\#585753).

%% file: theoretical-final.tex
\section{Theoretical background}\label{sec:theory}
In this section, we recall filtered $(\varphi,N)$-modules and Breuil and Kisin modules. We explain, in theory, how to calculate a finite height $\varphi$-module, over a ring larger than $\mathfrak S$, associated with a filtered $(\varphi,N)$-module (Theorem \ref{thm:connect}). In Section \ref{sec:explicit-family}, we carry this out in practice in a special case.

\subsection{Notations}
Let $k$ be a finite field and $W(k)$ be the Witt vectors over $k$. Set $K_0 = W(k)[1/p]$ and assume $K/K_0$ is a totally ramified extension of degree $e$.  Let $\Lambda_K$ be the ring of integers of $K$, $\pi \in \Lambda_K$ a uniformizer and $E= E(u)\in W(k)[u]$ its Eisenstein polynomial. Choosing $\pi_0 = \pi$ and $\pi_1,\pi_2,\dotsc$ a sequence in $\overline K$ such that $\pi_{i+1}^p = \pi_i$, we let $G_\infty$ be the absolute Galois group of $\varinjlim K(\pi_i)$. Let $\mathcal O \subseteq K_0[\![u]\!]$ be the rigid analytic functions on $|u| < 1$ and $\fS = W(k)[\![u]\!] \subseteq \mathcal O$. The action of $\varphi$ on $K_0[\![u]\!]$, by the Frobenius on $K_0$ and $\varphi(u) = u^p$, preserves $\fS \subseteq \mathcal O \subseteq K_0[\![u]\!]$.

We also choose $F/\Qp$ a finite extension, which will play the role of linear coefficients. In Section \ref{sec:comparison}, we assume $F$ contains a subfield isomorphic the Galois closure of $K$. We write $\Lambda \subseteq F$ for the ring of integers, $\mathfrak m_F \subseteq \Lambda$ the maximal ideal, and $\F$ for the residue field. Define $\fS_{\Lambda} = \fS\otimes_{\mathbb Z_p} \Lambda$ and $\mathcal O_F = \mathcal O\otimes_{K_0} F$. Extending $\varphi$ linearly, we have $\varphi$-stable subrings of $\mathfrak S_{\Lambda} \subseteq S_F \subseteq (K_0\otimes_{\mathbb Q_p} F)[\![u]\!]$, where $S_F = \fS [\![ \frac{E^p}{p}]\!] \otimes_{\Q_p} F$.

\subsection{Kisin modules} 
Let $R \subseteq (K_0\otimes_{\mathbb Q_p} F)[\![u]\!]$ be a $\varphi$-stable subring containing $E$. A $\varphi$-module over $R$ is a finite free $R$-module $M$ equipped with an injective $\varphi$-semilinear operator $\varphi_M : M \rightarrow M$. Let $\Mod_{R}^{\varphi}$ be the category of $\varphi$-modules over $R$ with morphisms being $R$-linear maps that commute with $\varphi$. For a $\varphi$-module $M$, write $\varphi^{\ast}M = R\otimes_{\varphi,R} M$, so $1\otimes \varphi_M$ defines an $R$-linear map $\varphi^{\ast}M \rightarrow M$ called the linearization of $\varphi$. For $h \geq 0$,  an element $M \in \Mod_R^{\varphi}$ has ($E$)-height $\leq h$ if its linearization has cokernel annihilated by $E^h$. The subcategory of $\varphi$-modules over $R$ with height $\leq h$ is denoted $\Mod_{R}^{\varphi,\leq h}$. A \emph{Kisin module} over $\fS_{\Lambda}$ with height $\leq h$ is an object in $\Mod_{\fS_{\Lambda}}^{\varphi,\leq h}$.

Let $\mathrm{MF}^{\varphi, N}_F$ be the category of \emph{positive} filtered $(\varphi, N,K,F)$-modules, which we shorten to just filtered $(\varphi,N)$-modules over $F$ (see \cite[Section 3.1.1]{BreuilMezard-Multiplicities}). For $D \in \MF^{\varphi,N}_F$ set $D_K = K\otimes_{K_0} D$; here, positive means $\Fil ^0 D_K = D_K$. Let $\mathrm{Rep}_{F}^{\mathrm{st}, h}$ be the category of $F$-linear semi-stable representations $V$ of $G_K$ whose Hodge-Tate weights lie in $\{0 , \dots, h\}$. Then, there exists a fully faithful, contravariant, functor
$$
D_{\st}^{\ast} : \Rep_F^{\mathrm{st},h} \rightarrow \MF_F^{\varphi,N} 
$$
whose image is the subcategory of weakly-admissible filtered $(\varphi,N)$-modules over $F$ (see \cite{Fontaine-RepresentationSemiStable,ColmezFontaine-WAdmImpliesAdm} and \cite[Corollaire 3.1.1.3]{BreuilMezard-Multiplicities}). For  $V \in \Rep_F^{\mathrm{st},h}$ and $T \subseteq V$ a $G_\infty$-stable and $\Lambda$-linear lattice there exists, by \cite[Theorem 5.4.1]{Liu-Torsion}, a canonical Kisin module $\mathfrak M = \mathfrak M(T)$ over $\mathfrak S_{\Lambda}$ with height $\leq h$. Naturally, we say a Kisin module $\mathfrak M$ is associated with $V$ if $\mathfrak M = \mathfrak M(T)$ for some $T$. By \cite[Corollary 2.3.2]{BergdallLevin-BLZ}, the semi-simple mod $p$ representation $\overline V$ can be determined from any associated Kisin module.

One category that intervenes in determining an $\mathfrak M$ associated with $V \in \Rep_F^{\mathrm{st},h}$ is the category of $(\varphi,N_{\nabla})$-modules over $\mathcal O_F$ (\cite{Kisin-FCrystals}). Let $\lambda = \prod_{n\geq 0} \varphi^{n}(E(u)/E(0)) \in \mathcal O_F$. An object $\mathcal M_{\mathcal O_F} \in \Mod_{\mathcal O_F}^{\varphi,N_{\nabla}}$ is a finite height $\varphi$-module over $\mathcal O_F$ equipped with a differential operator $N_{\nabla}$ lying over $-u\lambda{d \over du}$ on $\mathcal O_F$ and satisfying $N_{\nabla} \varphi = p{E(u)\over E(0)}\varphi N_{\nabla}$. By \cite[Theorem 1.2.15]{Kisin-FCrystals}, we have quasi-inverse equivalences of categories
\begin{equation}\label{eqn:MFMod-equiv}
\xymatrix{
\MF_F^{\varphi,N} \ar@/_/[r]_-{\underline{\mathcal M}_{\mathcal O_F}} & \Mod_{\mathcal O_F}^{\varphi,N_{\nabla}} \ar@/_/[l]_-{\underline D_{\mathcal O_F}}.
}
\end{equation}

For $s > 0$, write $\mathcal O_s$ for the $\mathcal O_F$-algebra of rigid analytic functions converging on $|u|<p^{-s}$.

\begin{prop}\label{prop:reductions-first}
Suppose $\mathfrak M \in \Mod_{\mathfrak S_{\Lambda}}^{\varphi,\leq h}$, $V \in \Rep_F^{\st,h}$, and $s$ is such that $1/pe < s < 1/e$ and $\mathfrak M \otimes_{\mathfrak S_\Lambda} \mathcal O_s \cong \underline{\mathcal M}_{\mathcal O_F}(D_{\st}^{\ast}(V)) \otimes_{\mathcal O_F} \mathcal O_s$ in  $\Mod_{\mathcal O_s}^{\varphi,\leq h}$. Then, $\mathfrak M = \mathfrak M(T)$ for some $T \subseteq V$ as above.
\end{prop}
\begin{proof}
Since $s < 1/e$, $\pi$ lies in the disc $|u| < p^{-s}$. Since $\mathfrak M \otimes_{\mathfrak S_\Lambda} \mathcal O_s \cong \underline{\mathcal M}_{\mathcal O_F}(D_{\st}^{\ast}(V)) \otimes_{\mathcal O_F} \mathcal O_s$, \cite[Corollary 2.2.5]{BergdallLevin-BLZ} implies that $\mathcal M_{\mathcal O_F} := \mathfrak M \otimes_{\mathfrak S_{\Lambda}} \mathcal O_F$ is canonically an object in $\Mod_{\mathcal O_F}^{\varphi,N_{\nabla}}$. Then, \cite[Theorem 5.4.1]{Liu-Torsion} implies that there exists a $V' \in \Rep_F^{\st,h}$ such that $\mathfrak M=\mathfrak M(T)$ for a lattice $T \subseteq V'$ for some $T$. We claim that $V \cong V'$. Indeed, since $1/pe < s < 1/e$, the definition of $\underline D_{\mathcal O_F}(\mathcal M_{\mathcal O_F})$ in \cite[Section 1.2.5-7]{Kisin-FCrystals} depends only the finite height $\varphi$-module $\mathcal M_{\mathcal O_F} \otimes_{\mathcal O_F} \mathcal O_s$ over $\mathcal O_s$. Thus,
 we have 
$$
D_{\st}^{\ast}(V') \cong \underline D_{\mathcal O_F}(\mathcal M_{\mathcal O_F}) \cong \underline D_{\mathcal O_F}(\underline{\mathcal M}_{\mathcal O_F}(D_{\st}^{\ast}(V))) \cong D_{\st}^{\ast}(V).
$$
Since $D_{\st}^{\ast}$ is fully faithful, we have $V \cong V'$, completing the proof.
\end{proof}

\begin{rmk}\label{rmk:adding-coefficients}
To be accurate, the equivalence \eqref{eqn:MFMod-equiv} is constructed in \cite{Kisin-FCrystals} only when $F = \mathbb Q_p$. We use multiple references with the same technical limitation. We pause to detail one approach to resolving the issue. Later, we omit details for other functors.

First, we may define the functors $\underline D_{\mathcal O_F}$ and $\underline{\mathcal M}_{\mathcal O_F}$ using the same formulas as \eqref{eqn:MFMod-equiv}, or, equivalently, we can define them by forcing the diagram
\begin{equation*}
\xymatrix{
\MF_F^{\varphi,N}\ar[d]_-{\text{forget}} \ar@/_/[r]_-{\underline{\mathcal M}_{\mathcal O_F}} & \Mod_{\mathcal O_F}^{\varphi,N_{\nabla}} \ar@/_/[l]_-{\underline D_{\mathcal O_F}} \ar[d]^-{\text{forget}} \\
\MF_{\mathbb Q_p}^{\varphi,N} \ar@/_/[r]_-{\underline{\mathcal M}_{\mathcal O}} & \Mod_{\mathcal O}^{\varphi,N_{\nabla}} \ar@/_/[l]_-{\underline D_{\mathcal O}}
}
\end{equation*}
to commute. Here, the vertical arrows are the natural forgetful functors and the bottom arrows are as in \cite{Kisin-FCrystals}, where they are proved to be quasi-inverses. If $\mathcal M_{\mathcal O_F} \in \Mod_{\mathcal O_F}^{\varphi,N_{\nabla}}$, we thus have a natural isomorphism $\alpha : \underline{\mathcal M}_{\mathcal O_F} (\underline D_{\mathcal O_F}(\mathcal M_{\mathcal O_F})) \cong \mathcal M_{\mathcal O_F}$ in $\Mod_{\mathcal O}^{\varphi,N_{\nabla}}$. Since multiplication by $x \in F$ defines an endomorphism of $\mathcal M_{\mathcal O_F}$ in $\Mod_{\mathcal O}^{\varphi,N_{\nabla}}$ and $\alpha$ is natural, we see $\alpha$ is an isomorphism in $\Mod_{\mathcal O_F}^{\varphi,N_{\nabla}}$. Thus, $\underline{\mathcal M}_{\mathcal O_F}$ is a left quasi-inverse to $\underline D_{\mathcal O_F}$. Proving $\underline D_{\mathcal O_F}$ is a right quasi-inverse to $\underline{\mathcal M}_{\mathcal O_F}$ is analogous.
\end{rmk}

\subsection{Breuil modules}\label{subsec:breuil-kisin-modules}
Let $S_{\rm{Br}}$ be the $p$-adic completion of  the divided power envelope of $W(k)[u]$ with respect to the ideal generated by $E$. Breuil (\cite{Breuil-Griffiths}) classically identified $\MF_{\mathbb Q_p}^{\varphi,N}$ with a category of filtered $(\varphi,N)$-modules over $S_{\rm{Br}}[\frac 1 p]$. We recall this, replacing $S_{\rm{Br}}$ with a simpler ring.

One extends the Frobenius $\varphi$ to $K_0 [\![u]\!]$ via $\varphi (u) = u ^p$. We define $N= -u{d\over du}$ on $K_0 [\![u]\!]$. Let $\widehat S _E$ be the $E$-completion of $W(k)[u][\frac 1 p]$. For a subring $R \subseteq \widehat  S_E$ and $j \geq 0$, set $\Fil ^j R = R \cap E ^j \widehat S_E$. In particular, we can take $R = S := W(k)[\![u, \frac{E^p}{p}]\!]$. As a subring of $K_0[\![u]\!]$,  $S$ is closed under $\varphi$ and $N$. We define $S_\Lambda = S \otimes_{\Z_p} \Lambda$ and $S_F = S \otimes_{\Z_p} F$, extending $\varphi$, $N$, and $\Fil^{\bullet}$ linearly. 

Clearly $S \subseteq S_{\rm {Br}}\subseteq \widehat S _E$ which are compatible with the $(u ,\frac{E^p}{p})$-topology on $S$, the $p$-adic topology on $S_{\rm Br}$ and the $(E)$-topology on $\widehat S_E$. One advantage $S$ enjoys over $S_{\rm{Br}}$ is that $\Fil ^ j S_F = E^j S_F$. To see this, note that any element $f \in \Fil^j S_F$ can be uniquely written in the form $f = \sum_{i} a_i(u){E^{pi}\over p^i}$ with $a_i(u) \in K_0[u]$ a polynomial of degree strictly less than $ep$ ($e$ is the degree of $E$). Then, when $j < pi$, we have $\frac{E^{p i -j}}{p ^i}=\frac{1}{p ^{i - l}} E^{p i - pl} (\frac {E^p}{p}) ^{l}$ with $l = \lfloor \frac{p i -j}{p} \rfloor$. In this situation, $i-l$ depends only on $j$, so factoring $E^j$ out of the expression for $f$ and examining the leftover summation, one sees at once that $f \in E^j S_F$. Note as well:\ $S_F$ is an $\mathcal O_F$-algebra, and $\varphi(E) = p{\fc}$ with ${\fc} \in S^\times$. In particular, $\varphi(\lambda) \in S^\times \subseteq S_F^\times$.

The category $\MF_{S_F}^{\varphi,N}$ of filtered $(\varphi,N)$-modules over $S_F$, or {\em Breuil modules} over $S_F$, are objects $(\mathcal D,\varphi_{\mathcal D}) \in \Mod_{S_F}^{\varphi}$ such that the linearization of $\varphi_{\mathcal D}$ is an isomorphism, and $\mathcal D$ is equipped with:
\begin{itemize}
	\item a decreasing filtration $\Fil^{\bullet} \mathcal D$ by $S_F$-submodules such that $\Fil^0\mathcal D = \mathcal D$ and $\Fil^i  S_F  \cdot \Fil^j \mathcal D \subseteq \Fil^{i+j} \mathcal D$  for all $i,j \geq 0$;
	\item an operator $N_{\mathcal D}: \mathcal D \rightarrow \mathcal D$ that acts as a derivation over $N$, and 
	\begin{itemize}
		\item $N_{\mathcal D} \varphi_{\mathcal D} = p \varphi_{\mathcal D} N_{\mathcal D}$, and
		\item $N_{\mathcal D}(\Fil^i \mathcal D) \subseteq \Fil^{i-1}\mathcal D$ for all $i \geq 1$.
	\end{itemize}
\end{itemize}
A morphism in $\MF_{S_F}^{\varphi,N}$ is an $S_F$-linear map equivariant for $\varphi$, $N$, and $\Fil^{\bullet}$.

We define a functor $\underline{\mathcal D}: \MF_F^{\varphi,N} \rightarrow \MF_{S_F}^{\varphi,N}$ as follows:
\begin{itemize}
	\item $\mathcal D := \underline{\mathcal D}(D) =  S_F \otimes_{K_0\otimes_{\mathbb Q_p} F} D$ as an $S_F$-module;
	\item $\varphi_{\mathcal D} = \varphi \otimes \varphi_{D}$;
	\item $N_{\mathcal D} = N\otimes 1 + 1 \otimes N_{D}$;
	\item $\Fil^0(\mathcal D) = \mathcal D$ and 
	$$
	\Fil^{i}(\mathcal D) = \{ x \in \mathcal D \mid N_{\mathcal D}(x) \in \Fil^{i-1}\mathcal D \text{ and } (\ev_\pi\otimes 1)(x) \in \Fil^i D_K\}
	$$
	for $i \geq 1$.
\end{itemize}
Here, $\ev_\pi : S_F \rightarrow F\otimes_{\mathbb Q_p} K$ is the scalar extension of $\ev_\pi : W(k)[u] \twoheadrightarrow \Lambda_K$, the evaluation at $\pi$ map.
\begin{thm}[Breuil]\label{thm:breuil-equivalence}
The functor $\underline{\D}: \MF_{F}^{\varphi,N} \rightarrow \MF_{S_F}^{\varphi,N}$ is an equivalence of categories.
\end{thm}
Breuil proves in \cite[Section 6]{Breuil-Griffiths} that $\underline{\mathcal D}$ is an equivalence of categories when $F = \mathbb Q_p$ and $S$ is replaced by $S_{\mathrm{Br}}$. That one can replace $S_{\mathrm{Br}}$ by $S$ is known to some, but there does not appear to be a reference. The only step in the proof of Breuil that requires honestly new justification is the following analog of \cite[Proposition 6.2.1.1]{Breuil-Griffiths}. (This version is even easier to prove.)

\begin{lemma}\label{lem:equivariant-analogue}
Let $\mathcal D \in \MF_{S_F}^{\varphi,N}$ and $D = \mathcal D / u \mathcal D$. Then, there exists a unique $F\otimes_{\mathbb Q_p} K_0$-linear $\varphi$-equivariant section $s : D \rightarrow \mathcal D$ of the reduction map.
\end{lemma}
\begin{proof}
First, suppose $F = \mathbb Q_p$ and let $(\widehat e_1, \dots, \widehat e_d)$ be an $S[\frac 1 p]$-basis of $\D$. Write $\varphi_\D (\widehat e_1 , \dots, \widehat e_d) = (\widehat e_1, \dots , \widehat e_d) X$ and set $X_0 = X \bmod u$. Then, $X \in p^k\Mat_{d}(S)$, $X_0^{-1} \in p^\ell\Mat_d(W(k))$, and $X X_0^{-1} \in I + up^m\Mat_d(S)$ for some $k,\ell,m \in \mathbb Z$. As in the proof of \cite[Proposition 6.2.1.1]{Breuil-Griffiths}, we need to show
$$Y_n : = X\varphi(X) \cdots \varphi ^n (X) \varphi ^n (X_0^{-1}) \cdots \varphi (X_0^{-1})   X_0^{-1}
$$ 
converges in $\Mat_d(S[\frac 1 p])$ as $n \rightarrow \infty$. But,  in the notation above,
$$
Y_n - Y_{n-1} \in \varphi^n(u) p^{n(k+\ell)+m} \Mat_d(S).
$$
Since $\varphi^n(u) p^{nr} \rightarrow 0$ in $S[\frac 1 p]$ for any fixed $r$, we see that $Y_n - Y_{n-1} \rightarrow 0$ in $\Mat_d(S[\frac 1 p])$, as needed. 

If $F \neq \mathbb Q_p$, the proof already given implies there exists a unique $K_0$-linear $\varphi$-equivariant section $s: D \rightarrow \mathcal D$. If $x \in F^\times$ then $x^{-1} s x$ also $K_0$-linear and $\varphi$-equivariant and thus $s$ is $F$-linear.
\end{proof}

\begin{proof}[Proof of Theorem \ref{thm:breuil-equivalence}]
 Define  $\underline D_{S_F} : \MF_{S_F}^{\varphi,N} \rightarrow \MF_{F}^{\varphi,N}$ as follows. Set $D = \underline D_{S_F}(\mathcal D)  = \mathcal D/u\mathcal D$ with its induced action of $\varphi$ and $N$. For $s$ in Lemma \ref{lem:equivariant-analogue}, $(\ev_{\pi}\otimes 1)\circ s : D \rightarrow \mathcal D/E \mathcal D$ induces a canonical isomorphism $D_K \cong \mathcal D/E \mathcal D$. The filtration $\Fil^i(D_K)$ is the pullback of the filtration on $\mathcal D / E \mathcal D$ defined as the image $\Fil^i(\mathcal D) \rightarrow \mathcal D / E\mathcal D$. The arguments in \cite{Breuil-Griffiths}, with Lemma \ref{lem:equivariant-analogue} replacing Proposition 6.2.1.1 of {\em loc.\ cit.}, show $\underline D_{S[\frac 1 p ]}$ and $\underline{\mathcal D}$ are quasi-inverses when $F = \mathbb Q_p$. In general, see Remark \ref{rmk:adding-coefficients}.
\end{proof}

\subsection{Comparison} \label{sec:comparison}

We now assume that $F$ contains a subfield isomorphic to the Galois closure of $K$ (see Lemma \ref{lem-Fil-free}). In practice, as in Sections \ref{sec:explicit-family} and \ref{sec:descent}, we take $K = \mathbb Q_p$ so this is no hindrance.

In the prior sections, we have described equivalences
\begin{equation}\label{eqn:Liu-equiv}
\xymatrix{
\Mod_{\mathcal O_F}^{\varphi,N_{\nabla}} \ar[r]^-{\cong} & \MF_F^{\varphi,N} \ar[r]^-{\cong} & \MF_{S_F}^{\varphi,N}.
}
\end{equation}
An analog of \cite[Corollary 3.2.3]{Liu-LatticesInSemiStable} allows for a description of the composition that, unfortunately, is not practical for calculations. Below, though, we explain how to determine $\underline{\mathcal M}_{\mathcal O_F}(D)\otimes_{\mathcal O_F} S_F$ as a $\varphi$-module over $S_F$ from $D$, up to determining $\mathcal D = \underline{\mathcal D}(D)$.  A key technical point, which follows from the next lemma, is that filtrations on Breuil modules over $S_F$ are always free, in contrast to the filtrations on objects in $\MF_F^{\varphi,N}$ (cf.\ \cite[Exemple 3.1.1.4]{BreuilMezard-Multiplicities}).

\begin{lemma}\label{lem-Fil-free}
Suppose that $\mathcal N$ is a finite free $S_F$-module and $\mathcal H \subseteq \mathcal N$ is an $S_F$-submodule such that $E^j \mathcal N \subseteq \mathcal H$ for some $j \geq 0$. Then, $\mathcal H$ is finite free over $S_F$.
\end{lemma}
\begin{proof}
We may assume $j = 1$. Indeed, consider the nested sequence $\mathcal H_i = \mathcal H + E^i \mathcal N$ of $S_F$-modules, which satisfy  $E \mathcal H_i \subseteq \mathcal H_{i+1} \subseteq \mathcal H_i$. By the $j=1$ case we deduce $\mathcal H_1 \subseteq \mathcal N$ is free, and then $\mathcal H_2$, and so on until $\mathcal H_j = \mathcal H$ is free. We may also assume $\mathcal N \cong S_F$. Indeed, if $ 0 \to \cN'' \to \cN \overset f \to  \cN' \to 0$ is an exact sequence of finite free $S_F$-modules, then $\cH' = f (\cH)$ and $\cH'' = \ker (f) \cap \cH$ satisfy $E \cN'' \subseteq \cH''$ and $E\cN' \subseteq \cH'$. So, if both $\cH''$ and $\cH'$ are free, then $\mathcal H \cong \mathcal H'' \oplus \mathcal H'$ is free as well.
%\footnote{\color{red} BL: Do you reallly know it is split?  I would think you might need freeness of $H'$ to get s splitting.  {\color{olive}JB: I re-wrote what I meant. Thanks. Check it out?}}

We have reduced to proving:\ if $I \subseteq S_F$ is an ideal containing $E$, then $I$ is free. Since $F$ contains a subfield isomorphic to the Galois closure of $K$, we may decompose $S_F = \prod_{\sigma \in \Hom(K_0, F)} S_{F, \sigma}$ where $S_{F, \sigma} = \Lambda[\![u, \frac{\sigma(E)^p}{p}]\!][\frac 1 p]$ is a domain.   The ideal $I$ decomposes as a product of ideals $I_\sigma$ such that $\sigma(E)S_{F,\sigma} \subseteq I_{\sigma}$. Since $\sigma(E)$ is non-zero, it suffices to show each $I_{\sigma}$ is principal.  Write $\Hom_{\sigma}(K,F)$ for the embeddings $\tau: K \rightarrow F$ lifting $\sigma$. Then, we have a canonical isomorphism
$$
S_{F, \sigma}/\sigma(E)S_{F,\sigma} \cong K \otimes_{K_0, \sigma} F \cong F^{\Hom_{\sigma}(K,F)}.
$$ 
So, $I_{\sigma}/\sigma(E)S_{F,\sigma} \cong F^T$ for some subset $T \subseteq \Hom_{\sigma}(K,F)$. But, $J_T = \prod_{\tau \in T} (u-\tau(\pi)) \cdot S_F$ also contains $\sigma(E)S_{F,\sigma}$ and $J_T/\sigma(E)S_{F,\sigma} \cong F^T$. Thus $I_{\sigma} = J_T$ is principal, completing the proof.
\end{proof}

We now consider an {\em ad hoc} category of ``Breuil modules without monodromy''. Let $\MF_{S_F}^{\varphi, h }$ denote the category whose objects are $(\mathcal D, \varphi_{\mathcal D}) \in \Mod_{S_F}^{\varphi}$ such that the linearization of $\varphi_{\mathcal D}$ is an isomorphism, and $\mathcal D$ is equipped with a finite {\em free} $S_F $-submodule  $\Fil^h \mathcal D \subseteq \D$ such that $\Fil^h S_F \cdot \mathcal D \subseteq \Fil ^h \D$.  By Lemma \ref{lem-Fil-free} there is a natural forgetful functor $
\MF_{S_F}^{\varphi,N} \rightarrow \MF_{S_F}^{\varphi,h}$. 

Now define $\underline{\D}' : \Mod_{S_F}^{\varphi, \leq h } \to \MF_{S_F}^{\varphi,h }$ by declaring $\underline{\mathcal D}'(\mathcal M) = S_F\otimes_{\varphi,S_F} \mathcal M$ as an $S_F$-module, and
\begin{itemize}
\item $\varphi_{\underline{\mathcal D}'(\cM)}= \varphi \otimes \varphi _{\cM}$, and
\item $\Fil ^h \underline{\D}'(\cM ) = \{ x \in  \underline{\D'}(\cM ) \mid (1 \otimes \varphi_{\cM})(x) \in \Fil ^h S_F\cdot \cM\}$.
\end{itemize} 
Since $E^h \underline{\mathcal D}'(\mathcal M) \subseteq \Fil^h\underline{\mathcal D}'(\mathcal M)$, Lemma \ref{lem-Fil-free} implies $\Fil ^h \underline{\D}'(\cM )$ is finite free over $S_F$.
 
\begin{prop}\label{prop-equiv2}
The functor $ \underline{\D}'$ is an equivalence. 
\end{prop}
\begin{proof} 
We first show $\underline{\D}'$ is fully faithful. Suppose $\cM$ and $\cM'$ are in $\Mod_{S_F} ^{\varphi , \leq h}$. Write $\D:=  \underline{\D}'(\cM )$ and $\D' :=  \underline{\D}'(\cM')$. Choose a basis $(e_1 , \dots , e_d)$ of $\cM$ and write $\varphi_{\mathcal M} (e_1, \dots , e_d ) = (e_1 , \dots , e_d ) A$ with $A \in \Mat_d(S_F)$. Since $\cM$ has height $\leq h$, there exists a matrix $B \in \Mat_d(S_F)$ such that $AB = BA = E^h I _d$. By assumption, $\Fil^h\D$ has basis $(\alpha_1,\dotsc,\alpha_d) = (\wt e_1 , \dots , \wt e_d ) B$ where $\wt e_i = 1 \otimes e_i \in \D$ compose a basis of $\D$. Similarly, we get $A', B'$ and $\wt e_i'$ from a basis $(e_1', \dots, e_{d'}')$ of $\mathcal M'$. 

Now suppose $f: \D \rightarrow \D'$ is a morphism in $\MF_{S_F}^{\varphi,h}$. We write $f(\wt e_1, \dots , \wt e_d)= (\wt e' _1, \dots , \wt e'_{d'}) X$ for $X \in \Mat_d(S_F)$. Since $f$ is $\varphi$-equivariant, we have $X \varphi (A) = \varphi (A') \varphi (X)$, and, since $f(\Fil^h\D)\subseteq \Fil^h\D'$, we have $X B = B' Y$ for some $Y \in \Mat_d(S_F)$. Using $AB = BA = E^h I_d  $ and $ A'B'= B'A'= E^h I_{d'}$, we see $\varphi(Y)\varphi(E^h) = X \varphi(E^h)$, and so $X = \varphi(Y)$ because $\varphi(E) \in S_F^\times$. It follows that $Y A = A' \varphi (Y)$. Define $\mathfrak f : \cM \to \cM'$ by $\mathfrak f (e_1, \dots , e_d)= (e'_1, \dots e'_{d'}) Y$. Then, $\mathfrak f$ is $\varphi$-equivariant and $f = \underline{\D}'(\mathfrak{f})$ since $X = \varphi(Y)$. This shows $\underline{\mathcal D}'$ is full, and since $Y$ determines $X$, we also see $\underline{\mathcal D}'$ is faithful.
	
	Now we prove  $ \underline{\D}'$ is  essentially surjective. Given a $\D \in \MF ^{\varphi, h}_{S_F}$, choose bases $(e_1 , \dots , e_d)$ of $\D$ and $(\alpha _1 , \dots , \alpha _d) $ of $\Fil ^h \D$. Write $ (\alpha _1, \dots , \alpha _d) = (e_1, \dots , e_d) B$ and $\varphi_{\D} (e_1, \dots,  e_d)= (e_1, \dots , e_d) X$ with $\det (X) \in S_F^\times$. Since $E^h \D \subseteq \Fil ^h \D$, there exists $A\in\Mat_d(S_F)$ such that $AB = BA = E^h I_d$. Since $\varphi(E)=p\fc \in S_F^\times$, we see $X\varphi(B) \in \GL_d(S_F)$, whereas $\varphi_{\D} (\alpha _1 , \dots , \alpha_d) = (e_1 , \dots , e_d) X \varphi (B)$. Thus $(f_1,\dotsc,f_d) = (e_1,\dotsc,e_d) X \varphi(B) p^{-h}\fc^{-h}$ is a basis of $\D$ and $\varphi_\D (\alpha_1 , \dots , \alpha _d) = (f_1, \dots , f_d) p ^ h \fc ^h$.  Finally, $(\alpha_1, \dots , \alpha _d) = (f_1, \dots , f_d) B ' $ where $B' = Y B $ and $Y= ( X\varphi (B) p^{-h} \fc ^{-h})^{-1}$, so there exists an $A'$ such that $A' B' = B' A' = E^h I_d$. Now define $\cM=\bigoplus_{i=1}^d S_F \mathfrak f_i$ and	set $\varphi_{\cM}({\mathfrak f}  _1 , \dots , {\mathfrak f} _d) = ({\mathfrak f}  _1 , \dots , \mathfrak f_d) A'$. Then, $\cM\in \Mod_{S_F} ^{\varphi, \leq h}$ and $\underline{\D'}(\cM)= \D$ (set $f_i = 1 \otimes \mathfrak f_i$).
	\end{proof}

We now reach the main theorem of this section, which provides a mechanism to calculate a finite height $\varphi$-module over $S_F$ explicitly from $D \in \MF_{F}^{\varphi,N}$. We write $\varphi(E) = p\fc$ with $\fc \in S^\times$ as above.
\begin{thm} \label{thm:connect} 
Suppose $D \in \MF_{F}^{\varphi,N}$. Write $\D' \in \MF^{\varphi,h}_{S_F}$ for the image of $\underline{\D}(D)$ under the natural forgetful functor and $\mathcal M = \underline{\mathcal M}_{\mathcal O_F}(D) \otimes_{\mathcal O_F} S_F$. Then, there is a natural isomorphism $\underline{\mathcal D}'(\mathcal M) \cong \mathcal D'.$

In particular, $\mathcal M$ is recovered from $D$ via the following steps:
\begin{enumerate}
\item Select $S_F$-bases $(e_1,\dotsc,e_d)$ of $\mathcal D = \underline{\mathcal D}(D)$ and $(\alpha_1,\dotsc,\alpha_d)$ of $\Fil^h \D$.
\item Let $\varphi_\D (e_1, \dots , e_d) = (e_1 , \dots, e_d) X$ and $(\alpha_1, \dots, \alpha_d) = (e_1, \dots , e_d) B$ with $X,B \in \Mat_d(S_F)$.
\item Then, $\cM$ has an $S_F$-basis $(\mathfrak f_1,\dotsc,\mathfrak f_d)$ in which $\varphi_{\cM}({\mathfrak f}  _1 , \dots , {\mathfrak f} _d) = ({\mathfrak f}  _1 , \dots , \mathfrak f_d) A$, where
$$  
A = E^h B^{-1} X \varphi (B) p^{-h}\fc^{-h}.
$$
\end{enumerate}
\end{thm}
\begin{proof}
To start, once the isomorphism $\underline{\mathcal D}'(\mathcal M) \cong \mathcal D'$ is justified, the ``in particular'' follows by tracing through the second half of the proof of Proposition \ref{prop-equiv2}.

For $\mathcal M_{\mathcal O_F} \in \Mod_{\mathcal O_F}^{\varphi,N_{\nabla}}$ we define $\mathcal D = \underline{\mathcal D}_{\mathcal O_F}(\mathcal M_{\mathcal O_F}) = S_F \otimes_{\varphi,\mathcal O_F} \mathcal M_{\mathcal O_F}$, which is a finite free $S_F$-module, and equip it with the following structure of a Breuil module over $S_F$:
\begin{itemize}
\item $\varphi_{\mathcal D} = \varphi \otimes \varphi_{\mathcal M}$;
\item $N_{\mathcal D} = N \otimes 1 + {p\over \varphi(\lambda)} \otimes N_{\nabla}$;
\item $\Fil^i(\mathcal D) = \{x \in \mathcal D \mid (1 \otimes \varphi_{\mathcal M})(x) \in \Fil^i S_F \otimes_{\mathcal O_F} \mathcal M_{\mathcal O_F} \}$.
\end{itemize}
Following the proof of \cite[Proposition 3.2.1]{Liu-Torsion}, replacing $S$ by $S_{\mathrm{Br}}$ and adding linear $F$-coefficients, we see $\underline{\mathcal D}_{\mathcal O_F}: \Mod_{\mathcal O_F}^{\varphi,N_{\nabla}} \rightarrow \MF_{S_F}^{\varphi,N}$ defines a functor. Moreover, if $\mathcal M_{\mathcal O_F}$ has height $\leq h$, then
\begin{equation*}
\underline{\mathcal D}_{\mathcal O_F}(\mathcal M_{\mathcal O_F}) \cong \underline{\mathcal D}'(\mathcal M_{\mathcal O_F}\otimes_{\mathcal O_F} S_F)
\end{equation*}
in the category $\MF_{S_F}^{\varphi,h}$. Thus, it remains to show that  $\underline{\mathcal D}_{\mathcal O_F}$ makes the diagram of functors
\begin{equation}\label{eqn:triangle-diagram}
\xymatrix{
\MF^{\varphi,N}_F \ar[r]^-{\underline{\mathcal D}} & \MF_{S_F}^{\varphi,N}\\
\Mod_{\mathcal O_F}^{\varphi,N_{\nabla}} \ar[u]^-{\underline D_{\mathcal O_F}} \ar[ur]_-{\underline{\mathcal D}_{\mathcal O_F}}
}
\end{equation}
commute as well. (In particular, $\underline{\mathcal D}_{\mathcal O_F}$ is an equivalence.) It is enough to check this when $F = \mathbb Q_p$ (by Remark \ref{rmk:adding-coefficients}). In that case, if $S$ is replaced by $S_{\mathrm{Br}}$, this is the statement of \cite[Corollary 3.2.3]{Liu-Torsion}. The proof in {\em loc.\ cit.}\ goes through here with only one adjustment. Namely, the isomorphism $S_{\mathrm{Br}}[\frac 1 p]\otimes_{K_0} \underline D_{\mathcal O}(\mathcal M_{\mathcal O}) \cong S_{\mathrm{Br}}[\frac 1 p]\otimes_{\varphi,\mathcal O} \mathcal M_{\mathcal O}$ implicit in the first two displayed equations of {\em loc.\ cit.}\ needs to have $S_{\mathrm{Br}}$ replaced by $S$.  To make this adjustment, consider the map $\xi: \mathcal O \otimes_{K_0} \underline{D}(\mathcal M_{\mathcal O}) \rightarrow \mathcal M_{\mathcal O}
$ constructed in \cite[Lemma 1.2.6]{Kisin-FCrystals}. Thus $\xi$ is a $\varphi$-equivariant injection with cokernel annihilated by $\lambda^h$ for some $h \geq 0$. From the diagram in the middle of the proof of {\em loc.\ cit.}\ we have $\xi$ factors
\begin{equation}\label{eqn:comm-triang}
\xymatrix{
\mathcal O \otimes_{K_0} \underline D(\mathcal M_{\mathcal O}) \ar[d] \ar[r]^-{\xi} & \mathcal M_\cO\\
\mathcal O \otimes_{\varphi,\mathcal O} \mathcal M_\cO . \ar[ur]_-{1\otimes \varphi}
}
\end{equation}
We deduce the vertical arrow in \eqref{eqn:comm-triang} has cokernel annihilated by $\varphi(\lambda)^h$.  Since $\varphi(\lambda) \in S^\times$, we have
\begin{equation*}
S[1/p] \otimes_{K_0} \underline D(\mathcal M_{\mathcal O}) \overset{1\otimes \xi}{\cong} S[1/p] \otimes_{\varphi,\mathcal O} \mathcal M_{\mathcal O}.
\end{equation*}
This completes the proof.
\end{proof}

\begin{rmk}\label{rmk-fil-ev} The above proof makes it clear to see that for $D \in \MF_F^{\varphi,N}$ and $\mathcal D = \underline{\mathcal D}(D) \in {\MF}_{S_F}^{\varphi,N}$, the map $\ev_\pi$ induces an isomorphism $\Fil^{i+ 1}\cD / E \Fil ^i \cD \cong \Fil^{i +1} D_K$. Indeed, since $\ev_{\pi}(\Fil^{i +1}\D) = \Fil ^{i +1}D_K$, it suffices to show that $E \D \cap \Fil ^{i +1}\D = E \Fil ^i \D$. Pick $y = E  x\in \Fil ^{i +1} \D$ with $x\in \D$. The proof of the theorem, especially the fact that \eqref{eqn:triangle-diagram} commutes, shows that  
$$
\Fil^{i +1}(\mathcal D) = \{x \in \mathcal D \mid (1 \otimes \varphi_{\mathcal M})(x) \in \Fil^{i +1} S_F \otimes_{\mathcal O_F} \mathcal M_{\mathcal O_F} \}.$$ 
Thus, we see that $(1 \otimes \varphi_{\mathcal M})(Ex)= E (1 \otimes \varphi_{\mathcal M})(x)\in  \Fil^{i +1} S_F \otimes_{\mathcal O_F} \mathcal M_{\mathcal O_F}$. Since $\Fil^n S_F = E^n S_F$, it is clear that 
$(1 \otimes \varphi_{\mathcal M})(x)\in  \Fil^{i} S_F \otimes_{\mathcal O_F} \mathcal M_{\mathcal O_F}$ and hence $x \in \Fil ^ i \D $ as required. (Compare with the end of the proof of \cite[Proposition 3.2.1]{Liu-LatticesInSemiStable}.)
\end{rmk}

\begin{exam} \label{ex-crys} 
Suppose $K = \Q_p$ and $V$ is crystalline. By \cite{Laffaille}, $D= D_{\st}^{\ast}(V)$ admits a strongly divisible lattice  $(M , \Fil ^i M, \varphi _i)$. More precisely, there exists an $F$-basis $(e_1, \dots , e_d)$ of $D$ and integers $0 =n_0 \leq n_1 \leq \cdots \leq n_h \leq d$ such that $\Fil ^ i D : = \bigoplus_{j\geq n_i } F e_j $, and $\varphi (e_1,\dots , e_d) = (e_1, \dots , e_d) X P$ where $X \in \GL_d (\Lambda)$ and $P$ is a diagonal matrix whose $ii$-th entry is $p^{s_i}$ where $s_i= \max\{j \mid n_j \leq i\}= \max \{ j \mid e_i \in \Fil ^j D\}$. Since $N = 0$ on $D$,  we easily compute that $\Fil ^h  \D$ admits a basis 
$(e_1,\dots , e_d) B$ where $B$ is the diagonal matrix with $(i,i)$-th entry is $E^{h- s_i}$ (cf.\ Section \ref{subsec:filtration} below). By the steps outlined in Theorem \ref{thm:connect} , using the basis $1 \otimes e_i \in \D$ we see the matrix of $\varphi$ on $\cM$ is given by $A=E^h B^{-1} X P \varphi(B) p ^{-h} \fc ^{-h}$, where $A= \Lambda X C$, and $\Lambda$ is a diagonal matrix with $(i,i)$-th entry is $E^{s_i}$ and $C$ is a diagonal matrix with $(i,i)$-th entry is $\fc^{-s_i}$.  
\end{exam}

%% file: explicit-final.tex
\section{An explicit determination of a Breuil module}\label{sec:explicit-family}

In this section, we assume $K = \mathbb Q_p$. We choose $\pi = -p$, so $E(u) = u+p$. We keep $F/\mathbb Q_p$ as a linear coefficient field and recall $\Lambda$ is its ring of integers. In Section  \ref{subsec:filtered-phiN}, we explain the definition of the filtered $(\varphi,N)$-module $D_{h+1,\mathcal L} \in \MF_{F}^{\varphi,N}$, for $h \geq 1$ and $\mathcal L \in F$, discussed in the introduction. Let $\mathcal M_{h+1,\mathcal L}=\underline{\mathcal M}_{\mathcal O_F}(D_{h+1,\mathcal L}) \otimes_{\mathcal O_F} S_F \in \Mod_{S_F}^{\varphi,\leq h}$. The ultimate goal (Theorem \ref{thm:initialSF-basis}) is to describe the matrix of $\varphi$ in a certain trivialization $\mathcal M_{h+1,\mathcal L} \cong S_F^{\oplus 2}$, at least if $\mathcal L \neq 0$. We begin by describing the Breuil module $\mathcal D_{h+1,\mathcal L} = \underline{\mathcal D}(D_{h+1,\mathcal L})$.

\subsection{The filtration on some rank 2 Breuil modules}\label{subsec:filtration}
In order to minimize notation, in this subsection, we let $D \in \MF_{F}^{\varphi,N}$ be any 2-dimensional filtered $(\varphi,N)$-module with Hodge--Tate weights $0 < h$. We also choose any basis $(f_1,f_2)$ for $D$ such that $\Fil^h D = Ff_2$. We  write $N_D(f_1,f_2) = (f_1,f_2)\begin{smallpmatrix}a & b \\ c & d\end{smallpmatrix}$ with  $\begin{smallpmatrix}a & b \\ c & d\end{smallpmatrix} \in \Mat_2(F)$. (Compare with Lemma \ref{lem:MF-basis}.) 

Set $\mathcal D = \underline{\mathcal D}(D) = S_F\otimes_F D$. For $f \in D$ we write $\widehat f = 1 \otimes f \in \mathcal D$.  In particular, $\mathcal D$ is a free $S_F$-module with basis $(\widehat f_1, \widehat f_2)$.  Recall that $\Fil^i\mathcal D$ is defined by $\Fil^0\mathcal D = \mathcal D$ and, for $i \geq 1$,
\begin{equation*}
\Fil^i\mathcal D = \{x \in \mathcal D \mid N_{\D}(x) \in \Fil^{i-1}\mathcal D \text{ and } \ev_{\pi}(x) \in \Fil^i D\}.
\end{equation*}
When $i = 1$, the condition $N_{\D}(x) \in \Fil^{0}\mathcal D = \mathcal D$ is a tautology. So, $\Fil ^1 \D = S_F \widehat{f}_2 + S_F E\widehat{f}_1$.

\begin{prop}\label{prop:filtration-equal} 
There exists $x_1,\dotsc,x_{h-1}\in F$ such that, if $0 \leq i \leq h$, then  
$$
\Fil^i\mathcal D = S_F\cdot\bigl(\widehat f_2 + (\sum_{j=1}^{i-1} x_jE^j)\widehat f_1\bigr) + S_F\cdot E^i \widehat f_1.
$$
\end{prop}
\begin{proof} 
Assume by induction on $0 \leq i < h$, that there exists $x_1,\dotsc,x_{i-1} \in F$ such that for each $0 \leq j \leq i$ we have $\Fil^j\D = S_F\cdot \widehat f_2^{(j)} + S_F\cdot \widehat f_1$, where $\widehat f_2^{(j)} = \widehat f_2 + (\sum_{m=1}^{j-1} x_m E^m)\widehat f_1$. Setting $\widehat f_2^{(0)} = \widehat f_2^{(1)} = \widehat f_2$ handles the case $i = 0$ and $i = 1$. So, suppose $1 \leq i < h$. 

For the $(i+1)$-st case, we first define $x_i \in F$. By induction, $N_{\D}(\widehat f_2^{(i)}) \in \Fil^{i-1}\D = S_F\widehat f_2^{(i-1)} + S_F E^{i-1}\widehat f_1$. Since $\widehat f_2^{(i-1)} = \widehat f_2^{(i)} - x_{i-1}E^{i-1} \widehat f_1$, we can write
$$
N_{\D}(\widehat f_2^{(i)}) = d_i \widehat f_2^{(i)} + b_i E^{i-1}\widehat f_1
$$
for some $d_i, b_i \in S_F$ (cf.\ Lemma \ref{lem-ind-ab} below). Set $x_i = b_i(\pi)/i\pi$, and then set $\widehat f_2^{(i+1)} = \widehat f_2^{(i)} + x_i E^i \widehat f_1$. Since $2\leq i+1 \leq h$, we have $\Fil^{i+1}D = F f_2$. Thus, $\ev_{\pi}(\widehat f_2^{(i+1)}) = \widehat f_2 \in \Fil^{i+1}D$. Further,
\begin{align} \label{e1}
N_{\D}(\widehat{f}^{(i +1)} _2) &= N_{\D}(\widehat{f}^{(i)} _2)- x_i i u E^{i -1} \widehat{f}_1 + x_i E^i N_{\D}(\widehat{f}_1)\\
&= d_i\widehat f_2^{(i)} + (b_i-x_i i u)E^{i-1}\widehat f_1 + x_i E^i N_{\D}(\widehat f_1).\nonumber
\end{align}
Note, the last summand in \eqref{e1} lies in $\Fil^i S_F \cdot \D \subseteq \Fil^i \D$, while the first lies in $\Fil^i \D$. By definition we have $\ev_\pi(b_i - x_i i u) = 0$ and so the the middle summand also lies in $\Fil^i S_F \cdot \D \subseteq \Fil^i \D$. Thus $\widehat f_{2}^{(i+1)} \in \Fil^{i+1}\D$.

For a moment, define $F^{i+1}\D = S_F \widehat f_2^{(i+1)} + S_F E^{i+1} \widehat f_1 \subseteq \Fil^{i+1}\D$. We want to show equality. Since $E\widehat f_2^{(i)} = E\widehat f_2^{(i+1)} - x_i E^{i+1}\widehat f_1$, we in fact have
$$
E\Fil^i \D \subseteq F^{i+1}\D \subseteq \Fil^{i+1}\D.
$$
Since $\ev_\pi$ gives an isomorphism $\Fil^{i+1}\D/E\Fil^i\D \cong Ff_2$ by Remark \ref{rmk-fil-ev}, and $\ev_\pi(F^{i+1}\D) \neq 0$, we conclude the natural map $F^{i+1}\D/E\Fil^i\D \rightarrow \Fil^{i+1}\D/E\Fil^i\D$ is an isomorphism. Thus, $F^{i+1}\D = \Fil^{i+1}\D$.
\end{proof}

The proof of Proposition \ref{prop:filtration-equal} allows for explicit control of the scalars $x_j$ in terms of the monodromy matrix $\begin{smallpmatrix} a & b \\ c & d\end{smallpmatrix}$. For the next two results, we explain this by re-examining the proof.
\begin{lemma} \label{lem-ind-ab} For $1 \leq i \leq h-1$, let $d_i, b_i \in S_F$ be such that $
N_{\D}(\widehat f_2^{(i)}) = d_i \widehat f_2^{(i)} + b_iE^{i-1}\widehat f_1$. Then, $d_1 = d$, $b_1=b$, $x_1 = \frac{b}{\pi}$ and for $1 \leq i < h-1$
\begin{align*}
d_{i+1} &= d_i + c x_i E^i\\
b_{i+1} &= x_i(a-cz_i - d_i) + (b_i - x_i i u)/E \\
x_{i+1} &= \frac{b_{i+1}(\pi)}{(i+1)\pi}
\end{align*}
where $z_i = \sum_{j=1}^i x_j E^j$. 
\end{lemma} 
\begin{proof}
The values of $d_1$, $b_1$, and $x_1$ follow immediately from $\widehat f_2^{(1)} = \widehat f_2$ and $N_{\D}(\widehat f_2) = b\widehat f_1 + d\widehat f_2$. Next, by \eqref{e1} and because $N_{\D}(\widehat f_1) = a\widehat f_1 + c\widehat f_2$, we have
\begin{equation}\label{eqn:recursive-formula-db}
N_{\D}(\widehat{f}^{(i +1)} _2) = d_i \widehat{f}^{(i)}_2 + (b_i - x_i i u) E^{i-1} \widehat{f}_1 + x_i E^i  (a \widehat{f}_1 + c \widehat{f}_2) .
\end{equation}
We can write $\widehat f_2^{(i)} = \widehat f_2^{(i+1)} - x_i E^i \widehat f_1$ and, separately, $\widehat f_2 = \widehat f_2^{(i+1)} - z_i \widehat f_1$. Thus \eqref{eqn:recursive-formula-db} becomes
\begin{equation*}
N_{\D}(\widehat f^{(i+1)}_2) = (d_i + cx_iE^i) \widehat f_2^{(i+1)} + \bigl(-d_ix_iE^i + (b_i-x_i i u)E^{i-1} + x_iE^i(a-cz_i)\bigr)\widehat f_1.
\end{equation*}
Factoring $E^i$ out of the $\widehat f_1$-coefficient, the result is clear.
\end{proof}

\begin{exam} \label{exam:h3}  
Below, in Lemma \ref{lemma:z-estimates}, we will need an explicit calculation of the $x_i$ and $z$. This can be done using the recursive formulas above. The calculations we need, both of which are straightforward, are:
\begin{align*}
x_2 &= \frac{b}{2\pi^2}(a-d-1)\\
z_2(0) &= \frac{b}{2}(a-d-3).
\end{align*} 
(See Example \ref{exam:rmk-p=3-h=3-L}, also.)

%We have 
%\[
%b_2 = \frac{b}{\pi} (a - c \frac{b}{\pi} E - d) + (b - \frac{b}{\pi} u)/E = \frac{b}{\pi} (a - d -1 - c \frac{b}{\pi} E). 
%\]
%Since $x_2 = {b_2(\pi) \over 2 \pi}$, we get 
%\[
%x_2 = {b \over 2 \pi^2} (a - d - 1).
%\]
\end{exam}

\begin{lemma}\label{lem-coe-ab} 
Assume that $a-d \in \Lambda$ and $bc \in \Lambda$. Then, for $1 \leq i \leq h-1$, we have
\begin{equation*}
 v_p (x_i)+ v_p (i!)+ i \geq v_p (b).
\end{equation*}
\end{lemma}
\begin{rmk}\label{rmk:b=0}
The lemma is consistent with $b = 0$ since $x_i = 0$, for all $i$, in that case.
\end{rmk}
\begin{proof}[Proof of Lemma \ref{lem-coe-ab}]
Given $v \in \mathbb R$ we write 
\begin{equation*}
A_v = \{ \sum_{j\geq 0} y_j E^j \in F[u] \mid v_p(y_j) + v_p(j!) + j \geq v\}.
\end{equation*}
Note that $A_v$ is a subgroup of $F[u]$. Since $v_p((j+k)!) \geq v_p(j!) + v_p(k!)$ for all non-negative integers $j,k$ (because binomial coefficients are integers), we have $A_v A_w \subseteq A_{v+w}$, as well. In particular, $A_0$ is a ring containing $\Lambda$ as a subring and each $A_v$ is an $A_0$-module. 

The lemma is equivalent to $x_i E^i \in A_{v_p(b)}$ for all $1 \leq i \leq h-1$, but to show $x_iE^i \in A_{v_p(b)}$ it suffices to show $b_iE^{i-1} \in A_{v_p(b)}$. Indeed, $b_i E^{i-1} \in b_i(\pi)E^{i-1} + E^i F[u]$,
and so if $b_iE^{i-1} \in A_v$ (for any $v$) then $v_p(b_i(\pi)) + v_p((i-1)!) + i-1 \geq v$. Since $b_i(\pi) = x_i i \pi$, by definition, we would clearly have $v_p(x_i) + v_p(i!) + i \geq v$ as well.

We have reduced to showing $b_i E^{i-1} \in A_{v_p(b)}$ for $1 \leq i \leq h-1$. For $i = 1$, by Lemma \ref{lem-ind-ab}, we have $b_1 = b$ and so the claim is clear. Now assume that $b_jE^{j-1} \in A_{v_p(b)}$ for all $j \leq i$.  By the previous paragraph we have $x_j E^{j} \in A_{v_p(b)}$ for all $j \leq i$, and so $z_j \in A_{v_p(b)}$ for all $j \leq i$ (including $z_0$, which we define to be 0). By Lemma \ref{lem-ind-ab}, we have
\begin{multline}\label{eqn:bi+1}
b_{i+1}E^i = (a-cz_i - d_i) x_iE^i + (b_i - x_i i u)E^{i-1}\\
 = (a-d-c(z_i+z_{i-1}))x_iE^i + b_i E^{i-1} - x_i i \pi E^{i-1} - x_i i E^{i}.
\end{multline}
It is clear by induction that the final three summands are in $A_{v_p(b)}$. For the first summand, we  know  $z_i + z_{i-1} \in A_{v_p(b)}$. Since $v_p(c)+v_p(b) \geq 0$ and $a-d \in \Lambda$, we see $a-d-c(z_i+z_{i-1}) \in A_0$. Since $x_iE^i \in A_{v_p(b)}$, by induction, the first summand also lies in $A_{v_p(b)}$. Thus, $b_{i+1}E^i \in A_{v_p(b)}$.
\end{proof}

\subsection{Explicit filtered $(\varphi,N)$-modules}\label{subsec:filtered-phiN}

Now assume $F$ contains an element $\varpi$ such that $\varpi^2 = p$. For $\mathcal L \in F$ and $h\geq 1$, we define $D_{h+1,\mathcal L} = Fe_1 \oplus Fe_2 \in \MF_{F}^{\varphi,N}$ where, in the basis $(e_1,e_2)$,
\begin{align*}
\varphi &= \begin{pmatrix} \varpi^{h+1} & 0 \\ 0 & \varpi^{h-1} \end{pmatrix} 
&
N &= \begin{pmatrix} 0 & 0 \\ 1 & 0 \end{pmatrix}
&
\Fil^i D_{h+1,\mathcal L} &= \begin{cases} D_{h+1,\mathcal L} & \text{if $i \leq 0$;}\\
F\cdot(e_1 + \mathcal Le_2) & \text{if $1 \leq i \leq h$;}\\
\{0\} & \text{if $h < i$}.
\end{cases}
\end{align*}
See \cite[Exemple 3.1.2.2(iv)]{BreuilMezard-Multiplicities}. It is useful make a change of a basis. Set $a_p = \varpi^{h-1}+\varpi^{h+1}$.
\begin{lemma}\label{lem:MF-basis}
If $\mathcal L \neq 0$, then $(f_1,f_2)  = (-\varphi(e_1+\mathcal Le_2),e_1+\mathcal Le_2)$ is a basis of $D_{h+1,\mathcal L}$ in which
\begin{align*}
\varphi &= \begin{pmatrix} a_p & -1 \\ p^{h} & 0 \end{pmatrix}
&
N &= \frac{p}{\mathcal L(1-p)}\begin{pmatrix} 1 & \varpi^{-h-1} \\ \varpi^{h+1} & -1 \end{pmatrix}
&
\Fil^i D_{h+1,\mathcal L} &= \begin{cases} D_{h+1,\mathcal L} & \text{if $i \leq 0$;}\\
Ff_2 & \text{if $1 \leq i \leq h$;}\\
\{0\} & \text{if $h < i$}.
\end{cases}
\end{align*}
\end{lemma}
\begin{proof}
If $\mathcal L \neq 0$, then $e_1 + \mathcal Le_2$ is {\em not} an eigenvector of $\varphi$, so $(f_1,f_2)$ is a basis. We leave calculating the matrices for the reader.
\end{proof}
Now let $\mathcal D_{h+1,\mathcal L} = \underline{\mathcal D}(D_{h+1,\mathcal L})$ and $\mathcal M_{h+1,\mathcal L} = \underline{\mathcal M}_{\mathcal O_F}(D_{h+1,\mathcal L}) \otimes_{\mathcal O_F} S_F \in \Mod_{S_F}^{\varphi,\leq h}$.  Recall that $\fc = \varphi(E)/p \in S_F^\times$. Let $\lambda_- = \prod_{n\geq 0} \varphi^{2n+1}(E)/p$ and $\lambda_{++} = \varphi(\lambda_-)$.

\begin{thm}\label{thm:initialSF-basis}
If $\mathcal L \neq 0$, there exists a basis of $\mathcal M_{h+1,\mathcal L}$ in which the matrix of $\varphi$ is given by 
\begin{equation*}
A = \begin{pmatrix} (a_p - p^h z)\left({\lambda_-\over \lambda_{++}}\right)^h & -1 + \varphi(z)(a_p - p^h z) \\ E^h & E^h \varphi(z) \left({\lambda_{++}\over \lambda_-}\right)^h \end{pmatrix},
\end{equation*}
where $z = \sum_{j=1}^{h-1} x_j E^j \in F[E]$. Moreover, if $v_p(\mathcal L^{-1}) \geq -1$, then
\begin{equation}\label{eqn:xj-estimate}
v_p(x_j) \geq v_p(\mathcal L^{-1}) - {h-1\over 2} - v_p(j!) - j 
\end{equation}
for each $1 \leq j \leq h-1$.
\end{thm}
\begin{proof}
Let $(f_1,f_2)$ be the basis as in Lemma \ref{lem:MF-basis}. Set $\widehat f_1 = 1 \otimes f_1$ and $\widehat f_2 = 1 \otimes f_2$, elements of $\mathcal D_{h+1,\mathcal L}$, as before. Then, the matrix of $\varphi$ in the basis $(\widehat f_1,\widehat f_2)$ of $\mathcal D_{h+1,\mathcal L}$ is $X = \begin{smallpmatrix} a_p & -1 \\ p^h & 0\end{smallpmatrix}$. Moreover, Proposition \ref{prop:filtration-equal} implies that $\Fil^h \D_{h+1,\mathcal L} = S_F\alpha_1\oplus S_F\alpha_2$, where
$$
(\alpha_1,\alpha_2) = (\widehat f_1,\widehat f_2)\begin{pmatrix} E^h & z \\ 0 & 1 \end{pmatrix} =: (\widehat f_1,\widehat f_2)B
$$
for $z = \sum_{j=1}^{h-1} x_j E^j$ and some $x_j \in F$. Theorem \ref{thm:connect}  implies that $\mathcal M_{h+1,\mathcal L}$ has a basis in which the matrix of $\varphi$ is given by
\begin{equation}\label{eqn:initial-A}
A' = E^h B^{-1} X \varphi(B) p^{-h}\fc^{-h} = \begin{pmatrix} a - p^h z & p^{-h}\fc^{-h}\left(-1 + \varphi(z)(a_p - p^h z)\right) \\ E^h p^h & p^{-h}\fc^{-h}E^h p^h \varphi(z)\end{pmatrix}.
\end{equation}
Since $\lambda_-$ and $\lambda_{++}$ are units in $S_F$, we can replace $A'$ by $CA'\varphi(C^{-1})$ for $C= \begin{smallpmatrix} p^h \lambda_-^h & 0 \\ 0 & \lambda_{++}^h \end{smallpmatrix}$. A short calculation shows $A = C A' \varphi(C^{-1})$, completing the general proof.

Finally, if $v_p(\mathcal L^{-1}) \geq -1$, then the matrix of $N$ in Lemma \ref{lem:MF-basis} satisfies the hypotheses of Lemma \ref{lem-coe-ab}. So, the estimates \eqref{eqn:xj-estimate} follow from the $b$-entry of the monodromy matrix being
$$
b = {-p\over \varpi^{h+1}\mathcal L(1-p)} = {-1\over \varpi^{h-1}\mathcal L(1-p)}.
$$
This completes the proof.
\end{proof}

\begin{rmk}
An analogous calculation in the crystalline case, where $z = 0$ (see Remark \ref{rmk:b=0}), was made in \cite[Section 3]{BergdallLevin-BLZ}. The technique here, passing through the category $\MF_{S_F}^{\varphi,N}$, is different than {\em loc.\ cit.} The descriptions are the same, though. Compare with Example \ref{ex-crys}.
\end{rmk}

\begin{exam}\label{exam:rmk-p=3-h=3-L}
We need one ad hoc calculation in Lemma \ref{lemma:z-estimates} below. Let $h = 3$. By Example \ref{exam:h3}, the element $z$ in Theorem \ref{thm:initialSF-basis} satisfies $z(0) = \frac{b}{2}(a-d-3)$ where $\begin{smallpmatrix} a & b \\ c & d \end{smallpmatrix}$ is the monodromy matrix in Lemma \ref{lem:MF-basis}. For $p = h= 3$, plugging in the explicit matrix, we see $z(0) = {1\over 4\mathcal L}\left(\frac{1}{\mathcal L} + 1\right)$.
\end{exam}

%% file: descent-final.tex
\section{Descent and reductions}\label{sec:descent}
The goal in this section is to prove the main theorem of this article. Given $h \geq 1$ and $\mathcal L \in F$ we write $V_{h+1,\mathcal L}$ for the unique two-dimensional representation of $G_{\mathbb Q_p}$ such that $D_{\st}^{\ast}(V_{h+1,\mathcal L}) \cong D_{h+1,\mathcal L}$ where $D_{h+1,\mathcal L}$ is as in Section \ref{subsec:filtered-phiN}. Write $\overline V$ for the semi-simple reduction modulo $\mathfrak m_F$ of $V$. Let $\mathbb Q_{p^2}$ be the unramified quadratic extension of $\Qp$, $\chi$ the unramified quadratic character of $G_{\mathbb Q_{p^2}}$, and $\omega_2$ a niveau 2 fundamental character of $\mathbb Q_{p^2}$. Note that $\operatorname{Ind}_{G_{\mathbb Q_{p^2}}}^{G_{\mathbb Q_p}}(\omega_2^{h}\chi)$ has determinant $\omega^h$, where $\omega$ is the cyclotomic character, and its restriction to inertia is $\omega_2^h \oplus \omega_2^{ph}$.

\begin{thm}\label{thm:main-thm-in-text}
Assume $h \geq 3$ and $p \neq 2$. Then, if $\mathcal L$ satisfies
$$
v_p(\mathcal L^{-1}) > {h-1\over 2} - 1 + v_p((h-1)!),
$$
then $\overline V_{h+1,\mathcal L} \cong \operatorname{Ind}_{G_{\mathbb Q_{p^2}}}^{G_{\mathbb Q_p}}(\omega_2^{h}\chi)$.
\end{thm}

\begin{rmk}\label{rmk:small-p}
Our contribution toward Theorem \ref{thm:main-thm-in-text} is limited to $h \geq 4$ and $p = h = 3$. The case of $h = 3$ and $p \geq 5$ follows from the work of Breuil and M\'ezard. If we were to use the weaker bound $v_p(\mathcal L^{-1}) > {h-1\over 2} + v_p((h-1)!)$, then our calculation would also cover the cases of $h=2$ and $h=3$. See Remark \ref{remark:final-remark} for further explanations.
\end{rmk}
We plan to take the matrix of $\varphi$ acting on $\mathcal M_{h+1,\mathcal L}=\underline{\mathcal M}_{\mathcal O_F}(D_{h+1,\mathcal L}) \otimes_{\mathcal O_F} S_F$ as in Theorem \ref{thm:initialSF-basis} and replace it with a $\varphi$-conjugate defined over $\mathfrak S_{\Lambda}$ when $v_p(\mathcal L^{-1})$ satisfies the bound in theorem. This defines a Kisin module $\mathfrak M$ for $V_{h+1,\mathcal L}$ that allows us to calculate the reduction $\overline{V}_{h+1,\mathcal L}$. Despite our theorem being limited to $h \geq 3$, we will present many calculations only assuming $h \geq 2$, in order to later justify Remark \ref{rmk:small-p}. So, we assume without further comment that:
\begin{align}\label{eqn:blanket-assumptions}
p &\neq 2 \text{ and } h \geq 2; \\
v_p(\mathcal L^{-1}) &> {h-1\over 2} - 1 + v_p((h-1)!).\nonumber
\end{align}
We will clarify result-by-result where we need to limit to $h\geq 3$ or $h \geq 4$. Also, fix $z = \sum_{j=1}^{h-1} x_j E^j$ as in Theorem \ref{thm:initialSF-basis}. Note that by \eqref{eqn:blanket-assumptions}, we have $v_p(\mathcal L^{-1}) \geq -1$ so the estimates \eqref{eqn:xj-estimate} in Theorem \ref{thm:initialSF-basis} hold.

\subsection{Preparing for descent}\label{subsec:Rm-basis}
Consider the ring 
$$
R_2 = \{ f = \sum a_i u^i \in F[\![u]\!] \mid i + 2v_p(a_i) \rightarrow \infty \text{ as $i \rightarrow \infty$}\}.
$$
Thus $R_2$ is the $F$-Banach algebra of series converging on $|u| \leq p^{-1/2}$. We equip $R_2$ with the valuation $v_{R_2}(\sum a_i u^i) = \inf_i \{i + 2v_p(a_i)\}$. The canonical map $\mathcal O_F \hookrightarrow R_2$ factors through $S_F$ since $v_{R_2}(E^p/p) = p - 2 > 0$.  Finally, given $v \in \R$, we define additive subgroups $H_v^{\circ} \subseteq H_v \subseteq R_2$ by
\begin{align*}
H_v &= \{ f \in R_2 \mid v_{R_2}(f) \geq v\}; &
H_v^{\circ} &= \{ f \in R_2 \mid v_{R_2}(f) > v\}.
\end{align*}
For any $v$, $H_v$ and $H_v^{\circ}$ are stable under $\varphi$. In fact,  for any $j \geq 0$ we have that
\begin{equation}\label{eqn:stability}
\varphi(H_v \cap u^j R_2) \subseteq H_{v+j(p-1)} \cap u^{pj}R_2,
\end{equation}
and the same for $H^{\circ}_v$ replacing $H_v$. See \cite[Lemma 4.1.1]{BergdallLevin-BLZ}, for instance.

Our first lemma, concerning some entries of the matrix in Theorem \ref{thm:initialSF-basis}, is straightforward so we omit the proof (compare with \cite[Lemma 5.1.1]{BergdallLevin-BLZ}).

\begin{lemma}\label{lemma:lambda-assertion}
Let $\lambda_- = \prod_{n\geq 0} \varphi^{2n+1}(E)/p$ and $\lambda_{++} = \varphi(\lambda_-)$ be as in Theorem \ref{thm:initialSF-basis}. Then,
\begin{enumerate}[label=(\alph*)]
\item $\lambda_- \in 1 + H_{p-2}$ and $\lambda_{++} \in 1 + H_{p^2-2}$;\label{enum:lambda-assertion:1plus}
\item $\lambda_-, \lambda_{++}\in R_2^\times$;\label{enum:lambda-assertion:unit}
\item $v_{R_2}(\lambda_-^{\pm 1}) = 0 = v_{R_2}(\lambda_{++}^{\pm 1})$.\label{enum:lambda-assertion:valuation}
\end{enumerate}
\end{lemma}
We also prepare estimates for $z$. Note that by \eqref{eqn:blanket-assumptions} the estimate \eqref{eqn:xj-estimate} becomes
\begin{equation}\label{eqn:xs-assumption}
v_p(x_j) > v_p((h-1)!) - v_p(j!) - j - 1 \geq - j - 1.
\end{equation}
Recall, we write $a_p = \varpi^{h-1} + \varpi^{h+1}$. Thus, $v_{p}(a_p) = \frac{h-1}{2}$.
\begin{lemma}\label{lemma:z-estimates}
For $z = \sum_{j=1}^{h-1} x_j E^j$ as above, and $\nu = -1 + \varphi(z)(a_p - p^h z)$, we have 
\begin{multicols}{2}
\begin{enumerate}[label=(\alph*)]
\item $p^h z \in H_{h-1}^{\circ}$;\label{enum:z-estimates:p^hz}
\item $\varphi(z) \in H_{-2}^{\circ}$;\label{enum:z-estimates:phiz}
\item $\nu \in -1 + H_{h-3}^{\circ}$;\label{enum:z-estimates:nu}
\item If $h \geq 3$, then $\nu \in R_2^\times$.\label{enum:z-estimates:nu-unit}
\end{enumerate}
\end{multicols}
\noindent Furthermore, if $p =3$ and $h =3$, then $\phz(z) \in H^{\circ}_{-1}$ and $\nu \in -1 + H_{h-2}^{\circ} = -1 + H^{\circ}_{1}$.  
\end{lemma}
\begin{proof}
First, $v_{R_2}(E^j) = j$. By the ultrametric inequality and \eqref{eqn:xs-assumption}, we see
$$
v_{R_2}(z) > \inf\{2(-j-1) + j \mid 1 \leq j \leq h-1\} = -1-h.
$$
Part \ref{enum:z-estimates:p^hz} follows because $v_{R_2}(p^h) = 2h$. For \ref{enum:z-estimates:phiz},  note $v_{R_2}(\varphi(E)^j) = 2j$. Thus, using \eqref{eqn:xs-assumption},
\begin{equation*}
v_{R_2}(\varphi(z)) > \inf \{2(-j-1) + 2j \mid 1 \leq j \leq h-1 \} = -2.
\end{equation*}
Continuing, $\varphi(z)p^h z \in H_{h-3}^{\circ}$ by parts \ref{enum:z-estimates:p^hz} and \ref{enum:z-estimates:phiz} and, since $v_{R_2}(a_p) = h-1$, we have  $\varphi(z)a_p \in H_{h-3}^{\circ}$. This proves \ref{enum:z-estimates:nu}. Finally, part \ref{enum:z-estimates:nu-unit} follows from the geometric series and part \ref{enum:z-estimates:nu}.

Finally, suppose $p = h = 3$. By the argument for (c) above, it suffices to show $\varphi(z) \in H_{-1}^{\circ}$. We note $v_{R_2}(\varphi(E)^j - E(0)^j) \geq p + 2j - 2$ for any $j$. Thus, by \eqref{eqn:xs-assumption}
\begin{equation}\label{eqn:phi-diff}
v_{R_2}(\phz(z) - \varphi(z)(0))>  p + 2j - 2  - 2( j +1) = p - 4 = -1.
\end{equation}
But, by Example \ref{exam:rmk-p=3-h=3-L} we have $\varphi(z)(0) = z(0) = \frac{1}{4\mathcal L}\left(\frac{1}{\mathcal L} + 1\right)$. Since $v_p(\mathcal L^{-1}) > 0$, \eqref{eqn:phi-diff} then implies $v_{R_2}(\varphi(z)) > -1$ as we wanted.
\end{proof}

We now write $\mathcal M_2 = \mathcal M_{h+1,\mathcal L} \otimes_{S_F} R_2 \cong \underline{\mathcal M}_{\mathcal O_F}(D_{h+1,\mathcal L}) \otimes_{\mathcal O_F} R_2$. Thus, $\mathcal M_2 \in \Mod_{R_2}^{\varphi,\leq h}$. We also introduce some notation. Given $A \in \Mat_d(R_2)$ and $C \in \GL_d(R_2)$ we write $
C \ast_\varphi A = C\cdot A \cdot \varphi(C)^{-1}$. Thus, if $(e_1,e_2)$ is a basis of $\mathcal M_2$ and $A$ is the matrix of $\varphi_{\mathcal M_2}$ in that basis, then $C\ast_\varphi A$ is the matrix of $\varphi_{\mathcal M_2}$ in the basis $(e_1',e_2')$ is given by $(e_1',e_2') = (e_1,e_2)C^{-1}$.

\begin{prop}\label{prop:descent-almostready} Assume $h \geq 4$ or $p=h=3$.   Then, there exists a basis of $\mathcal M_2$ in which the matrix of $\varphi_{\mathcal M_2}$ is $\begin{smallpmatrix} G & -1 \\ E^h & 0 \end{smallpmatrix}$, where $
G \in (a_p - p^h z)\left({\lambda_-\over \lambda_{++}}\right)^h + H_h^{\circ}$.
\end{prop}
\begin{proof}
By Theorem \ref{thm:initialSF-basis}, there is a basis $(e_1,e_2)$ of $\mathcal M_2$ such that $\varphi_{\mathcal M_2}(e_1,e_2) = (e_1,e_2)A$, where
\begin{equation*}
A = \begin{pmatrix} (a_p - p^h z)\left({\lambda_-\over \lambda_{++}}\right)^h & -1 + \varphi(z)(a_p - p^h z) \\ E^h & E^h \varphi(z) \left({\lambda_{++}\over \lambda_-}\right)^h \end{pmatrix} = \begin{pmatrix} \mu & \nu \\ E^h & \eta \end{pmatrix},
\end{equation*}
where $\nu$ is as in Lemma \ref{lemma:z-estimates} and $\mu$ and $\eta$ are defined by the equality. Assume for now just that $h \geq 3$. Then, by Lemma \ref{lemma:z-estimates}\ref{enum:z-estimates:nu-unit}, $\nu \in R_2^\times$. Making a change of basis on $\mathcal M_2$, we replace $A$ by (note that $\mu\eta = (1+\nu)E^h$)
\begin{equation*}
A' = \begin{pmatrix} 1 & 0\\ -\eta/\nu & 1 \end{pmatrix} \ast_\varphi A = \begin{pmatrix} \mu + {\nu\varphi(\eta)\over \varphi(\nu)} & \nu \\ -E^h \nu^{-1} & 0 \end{pmatrix}.
\end{equation*}
Since $v_{R_2}(\nu +1) > 0$ by Lemma \ref{lemma:z-estimates}\ref{enum:z-estimates:nu}, we have $\nu(0) \in \Lambda^\times$. Thus $\nu_0 = \nu/\nu(0) \in 1 + (H_{h-3}^{\circ} \cap uR_2)$. By \eqref{eqn:stability}, we have $\varphi^k(\nu_0) \in 1 + H_{h-3+m_k}$ where $m_k \rightarrow \infty$ as $k \rightarrow \infty$. Thus, the infinite product $\nu_+ = \prod_{n \geq 0} \varphi^{2n}(\nu_0)$ converges in $R_2$. Set $\nu_- = \varphi(\nu_+)$, so $\nu_{\pm} \in 1 + H_{h-3}^{\circ} \subseteq R_2^\times$. We now change basis on $\mathcal M_2$ again to get a matrix $A''$ for $\varphi_{\mathcal M_2}$ given by
\begin{equation*}
A'' = \begin{pmatrix}  {-1\over \nu(0)} {\nu_- \over \nu_+} & 0 \\ 0 & {\nu_+ \over \nu_-}\end{pmatrix} \ast_{\varphi} A' = \begin{pmatrix} G & -1 \\ E^h & 0 \end{pmatrix},
\end{equation*}
where 
\begin{equation}\label{eqn:F-defn}
G = \left(\mu + {\nu\varphi(\eta)\over \varphi(\nu)}\right) {\nu_-^2 \over \nu_+ \nu_{++}}
\end{equation}
and $\nu_{++} = \varphi(\nu_-)$.  

To complete the argument, we justify $G \in \mu + H_h^{\circ}$. We already know $\nu_-^2/\nu_+\nu_{++} \in 1 + H_{h-3}^{\circ}$. The same is true for $\nu/\varphi(\nu)$. So,
\begin{equation}\label{eqn:nuphi(eta)/phi(nu)}
v_{R_2}\left({\nu\varphi(\eta)\over \varphi(\nu)}\right) \geq v_{R_2}(\varphi(\eta)) \geq v_{R_2}(\varphi(E)^h\varphi^2(z)),
\end{equation}
where we used Lemma \ref{lemma:lambda-assertion} to remove $\lambda_-$ and $\lambda_{++}$ from the estimate. We note $v_{R_2}(\varphi(E)^h) = 2h$ and $v_{R_2}(\varphi^2(z)) \geq v_{R_2}(\varphi(z)) > -2$, by \eqref{eqn:stability} and Lemma \ref{lemma:z-estimates}\ref{enum:z-estimates:phiz}. Thus from \eqref{eqn:nuphi(eta)/phi(nu)} we deduce that $v_{R_2}(\nu\varphi(\eta)/\varphi(\nu)) > 2h - 2 = 2(h-1)$.  We also note that $a_p-p^hz \in H_{h-1}$. Thus, $\mu \in H_{h-1}$ and so, returning to the definition \eqref{eqn:F-defn} of $\mu$ and $G$, we see
\begin{equation*}
G \in \bigl(\mu + H_{2(h-1)}^{\circ}\bigr)\cdot (1 + H_{h-3}^{\circ}) \subseteq \mu + H_{2h-4}^{\circ}+  H^{\circ}_{2(h-1)} = \mu + H_{2h-4}^{\circ}.
\end{equation*}
Now, if $h \geq 4$, then $2h-4 \geq h$ and so $G \in \mu + H_h^{\circ}$.  This completes the proof except if $p = h = 3$. In that case, Lemma \ref{lemma:z-estimates} shows $\nu \in -1 + H^{\circ}_1$, rather than $-1 + H_0^{\circ}$, from which we deduce
\begin{equation*}
G \in \bigl(\mu + H_{4}^{\circ}\bigr)\cdot (1 + H_{1}^{\circ}) \subseteq \mu + H_3^{\circ} = \mu + H_h^{\circ}
\end{equation*}
anyways. This completes the proof.
\end{proof} 

\subsection{Descent}\label{subsec:descent}
To descend to $\mathfrak S_{\Lambda}$, we use the algorithm from \cite[Section 4]{BergdallLevin-BLZ}. Write $T_{\leq d} : R_2 \rightarrow F[u]$ for the ``truncation'' operation $T_{\leq d}(\sum a_i u^i) = \sum_{i\leq d} a_i u^i$ and $T_{>d}(f) = f - T_{\leq d}(f)$. In the next two proofs, we will use the following principle:\ if $f \in R_2$ and $v_{R_2}(T_{\leq d}(f)) > d$ (for instance, if $v_{R_2}(f) > d$) then $T_{\leq d}(f) \in \mathfrak m_F[u]$.

\begin{prop}\label{prop:descent-prop} 
Suppose that $G \in R_2$ such that 
\begin{enumerate}[label=(\alph*)]
\item $G \in H_{h-1}$;\label{enum:descent-prop:g-ass}
\item $T_{>h}(G) \in H_{h-1}^{\circ}$;\label{enum:descent-prop:error-ass}
\item $T_{\leq h}(G) \in \mathfrak m_F[u]$.\label{enum:descent-prop:integral-ass}
\end{enumerate}
Then, given $A = \begin{smallpmatrix} G & -1 \\ E^h & 0 \end{smallpmatrix}$, there exists $C \in \GL_2(R_2)$  and $P \in \mathfrak m_F[u]$ such that $C \ast_{\varphi} A = \begin{smallpmatrix} P & -1 \\ E^h & 0 \end{smallpmatrix}$.
\end{prop}
\begin{proof}
Since $E^h \in u^h + H_{h+1}$, the assumption \ref{enum:descent-prop:g-ass} implies that
\begin{equation*}
A \in \begin{pmatrix} 0 & -1 \\ u^h & 0 \end{pmatrix} + \begin{pmatrix} H_{h-1} & 0 \\ H_{h + 1} & 0 \end{pmatrix}.
\end{equation*}
In the notation of \cite[Section 4.3]{BergdallLevin-BLZ}, set $a = 0$, $b = h$, $a' = {h \over 2} - {p-1 \over 2}$ and $b' = {h \over 2} + {p-1 \over 2}$, and $(c_0,c_h) = (-1,1)$. Since $h-1-a' = {h\over 2} - 1 + {p-1\over 2} \geq 1$, we see $A$ is $\gamma$-allowable with $\gamma = 1$ in the sense of \cite[Definition 4.3.1]{BergdallLevin-BLZ}. The error of $A$, in the same definition, is $\varepsilon = v_{R_2}(T_{>h}(G)) - a'$. By \cite[Theorem 4.3.7]{BergdallLevin-BLZ}, with $R=R_2$ in {\em loc.\ cit.}, there exists $C \in \GL_2(R_2)$ such that $A' = C\ast_{\varphi} A$ satisfies:
\begin{enumerate}[label=(\roman*)]
\item Evaluating at $u = 0$, we have $A'|_{u=0} = A|_{u=0}$.\label{enum:alg-app-u=0}
\item The matrix $A'$ is of the form $A' = \begin{smallpmatrix} P & -1 \\ f & 0 \end{smallpmatrix}$ with $P$ and $f$ polynomials of degree at most $h$.\label{enum:alg-app-truncation}
\item We have an estimate $v_{R_2}(P - T_{\leq h}(G)) \geq \varepsilon + a' + 1$.\label{enum:alg-app-estimate}
\end{enumerate}
(For the reader checking references, note that the role of $A$ versus $C$ is reversed in \cite{BergdallLevin-BLZ}.)

We claim $P \in \mathfrak m_F[u]$ and $f = E^h$, which would finish the proof of the proposition. To see $P \in \mathfrak m_F[u]$, we start by combining the estimate \ref{enum:alg-app-estimate} and the assumption \ref{enum:descent-prop:error-ass} in order to see that 
\begin{equation*}
v_{R_2}(P - T_{\leq h}(G)) \geq \varepsilon + a' + 1 =  v_{R_2}(T_{>h}(G)) +1 > h.
\end{equation*}
On the other hand, $P-T_{\leq h}(G)$ has degree at most $h$ by \ref{enum:alg-app-truncation} and so $P-T_{\leq h}(G) \in \mathfrak m_F[u]$, which implies $P \in \mathfrak m_F[u]$ by assumption \ref{enum:descent-prop:integral-ass}. 

To see $f = E^h$, we evidently have $f = \det(A') = rE^h$ for some $r \in R_2^\times$. In particular, $f$ has a root of multiplicity $h$ at $u = -p$. But, $f$ is a polynomial of degree at most $h$ by point \ref{enum:alg-app-truncation}, and by point \ref{enum:alg-app-u=0} we have $f(0) = E(0)^h$. It now follows quickly that $f = E^h$, since $F[\![u]\!]$ is a unique factorization domain.
\end{proof}
We now verify the $G$ from Proposition \ref{prop:descent-almostready} satisfies the hypothesis of Proposition \ref{prop:descent-prop}.
\begin{lemma}\label{lemma:Gintegral-properties}
Let $G \in (a_p - p^hz)\left({\lambda_-\over \lambda_{++}}\right)^h + H_{h}^{\circ}$. Then, 
\begin{enumerate}[label=(\alph*)]
\item $G \in H_{h-1}$,\label{enum:Gintegral-prop:G}
\item $T_{>h}(G) \in H_{h-1}^{\circ}$, and\label{enum:Gintegral-prop:>hG} 
\item $T_{\leq h}\left(G \right) \in \mathfrak m_F[u]$.\label{enum:Gintegral-prop:<hG}
\end{enumerate}
\end{lemma}
\begin{proof}
First, the conclusions depend only on $G \bmod H_h^{\circ}$, so we suppose $G = (a_p - p^h z)\left({\lambda_-\over \lambda_{++}}\right)^h$. Part \ref{enum:Gintegral-prop:G} follows from Lemmas \ref{lemma:lambda-assertion} and \ref{lemma:z-estimates}. For part \ref{enum:Gintegral-prop:>hG}, we first have, by Lemma \ref{lemma:lambda-assertion}\ref{enum:lambda-assertion:1plus}, that $a_p\left({\lambda_-\over \lambda_{++}}\right)^h \in a_p + a_p H_{p-2}$. So, $T_{>0}\big(a_p\left({\lambda_-\over \lambda_{++}}\right)^h\big) \in H_{h+p-3} \subseteq H_h$. On the other hand, by Lemma \ref{lemma:z-estimates}\ref{enum:z-estimates:p^hz} we have $p^hz \in H_{h-1}^{\circ}$. Thus we've shown in fact $T_{>0}(G) \in H_{h-1}^{\circ}$.

%Finally, we consider part \ref{enum:Gintegral-prop:<hG}.  Let $\fc = \varphi(E)/p =1 + {u^p\over p}$. One sees (cf.\ \cite[Equation (5.5)]{BergdallLevin-BLZ}) that
%\begin{equation*}
%v_{R_2}\left(\left({\lambda_-\over \lambda_{++}}\right)^h - \fc^h \right) \geq p^2 - 2 > 2.
%\end{equation*}
%By Lemma \ref{lemma:z-estimates}\ref{enum:z-estimates:p^hz}, $a_p-p^hz \in H_{h-1}$, and so \ref{enum:Gintegral-prop:<hG} holds if and only if it holds for $G = (a_p-p^h z)\fc^h$. 
%
%It is clear from direct calculation that $T_{\leq h}(a_p \fc^h)$ lies in $\mathfrak m_F[u]$. Moreover, since $p^h z \in H_{h-1}^{\circ}$ and $c^h \in H_0$, in order to justify $T_{\leq h}(p^h z c^h) \in \mathfrak m_F[u]$, we need only verify the the $u^h$-coefficient in $p^h z c^h = z\varphi(E)^h$ lies in $\mathfrak m_F$. Recall that $z = \sum_{j=1}^{h-1} x_jE^j$. For any given $1 \leq j \leq h-1$, we have
%\begin{equation}\label{eqn:modulo-check}
%\text{$u^h$-coefficient of $x_j E^j\varphi(E)^h$} \in \sum_{\substack{0\leq k \leq j\\ 0 \leq \ell \leq h \\ k + p\ell = h}} \mathbb Z  x_j p^{j-k}p^{h-\ell}.
%\end{equation}
%But, if $k \leq j < h$ and $k + p\ell = h$, then $\ell \geq 1$. By \eqref{eqn:xs-assumption}, we have for $(k,\ell)$ in the sum \eqref{eqn:modulo-check} that
%\begin{equation*}
%v_p(x_j p^{j-k}p^{h-\ell}) > -1 - k + h- \ell = - 1 + (p-1)\ell \geq p-2.
%\end{equation*}
%We conclude \eqref{eqn:modulo-check} lies in $\mathfrak m_F$, which is what we wanted.

Finally, we consider part \ref{enum:Gintegral-prop:<hG}. Since $E = u+p$, any $f \in S_\Lambda$ can be written $f= \sum\limits_{n = 0}^\infty \alpha_n \frac{u ^ n}{p^{\lfloor \frac n p  \rfloor}}$ with $\alpha_n \in \Lambda$. Let $f = {\lambda_-\over \lambda_{++}} \in S_\Lambda$, in particular. Since $v_p (a_p) = \frac{h -1}{2} > \lfloor \frac h  p  \rfloor$ unless $ p = h =3$ (or $p=2$, which we have excluded in \eqref{eqn:blanket-assumptions}), we see immediately that $T_{\leq h} (a_p f^h) \in \fm_F[u]$ except when $h = p = 3$. When $h = p$, however, 
\begin{equation*}
T_{\leq p}(f^p) = T_{\leq p}\left(\left(\sum_{n=0}^{p-1} \alpha_n u^n + \alpha_p \frac{u^p}{p}\right)^p\right) \in p\cdot \alpha_0^{p-1}\alpha_p \frac{u^p}{p} + \Lambda[u] \subseteq \Lambda[u].
\end{equation*}
Since $v_p(a_p) > 0$, we see $T_{\leq h}(a_pf^h) \in \mathfrak m_F[u]$ in every case.

By the prior paragraph, to show (c) it remains to show that $T_{\leq h} (p ^h z f^h ) \in \fm_F[u]$ as well. By definition, we can write $f^h = \sum\limits_{i = 0}^\infty \beta_i\frac { E^i}{p ^{\lfloor \frac i p\rfloor}}$ with $\beta _i \in \Lambda$ and recall $z= \sum_{j=1}^{h-1} x_jE^j$. Thus,
\begin{equation}\label{eqn:p^hzf^h}
p^h z f^h = \sum_{n=1}^\infty \left(\sum_{i+j=n} p^h x_j \beta_i p^{-\lfloor \frac i p\rfloor}\right) E^n.
\end{equation}
Using the binomial expansion of $E^n = (u+p)^n$ we see that the $u^m$-term of \eqref{eqn:p^hzf^h} is exactly equal to
$$
\sum_{n=m}^\infty \left(\sum_{i+j=n} p^h x_j \beta_i p^{-\lfloor \frac i p\rfloor}\right){n \choose m}p^{n-m}.
$$
We must show this has positive $p$-adic valuation for $m \leq h$. Since $\beta_i \in \Lambda$ and binomial coefficients are integers, it is enough to show that for all $m \leq h$ and $j < h$, if $n \geq m,j$ then
\begin{equation}\label{eqn:xj-inequality}
v_p(x_j) + h + n-m-{\lfloor \frac{n -j}{p}\rfloor} > 0.
\end{equation}
By \eqref{eqn:xs-assumption} we have $v_p(x_j) > -j-1$ and so
\begin{equation}\label{eqn:how-many-inequalities}
v_p(x_j) + h+n-m-{\lfloor \frac{n -j}{p}\rfloor} > h - m - 1+ n - j - {\lfloor \frac{n -j}{p}\rfloor}.
\end{equation}
But, the right-hand side of \eqref{eqn:how-many-inequalities} is non-negative. Indeed, when $h > m$, this is clear because $n\geq j$. When $h = m$, on the other hand, we have $n \geq m = h > j$. So, the right-hand side of \eqref{eqn:how-many-inequalities} in that case has the form $x - \lfloor x/p\rfloor -1$ with $x \geq 1$, which is also non-negative.
\end{proof}

\subsection{Proof of Theorem \ref{thm:main-thm-in-text}}\label{subsec:theorem-proof}
Finally, we give the proof of the main theorem:
\begin{quote}
{\it 
Assume that $h \geq 3$ and $p \neq 2$. Then, if $\mathcal L$ satisfies
$$
v_p(\mathcal L^{-1}) > {h-1\over 2} - 1 + v_p((h-1)!),
$$
then $\overline V_{h+1,\mathcal L} \cong \operatorname{Ind}_{G_{\mathbb Q_p^2}}^{G_{\mathbb Q_p}}(\omega_2^{h}\chi)$.
}
\end{quote}
\begin{proof}[Proof of Theorem \ref{thm:main-thm-in-text}]

First, if $h = 3$ and $p \geq 5$, then the assumption is that $v_p(\mathcal L) < 0$. The verification that $\overline V_{4,\mathcal L} \cong \operatorname{Ind}_{G_{\mathbb Q_p^2}}^{G_{\mathbb Q_p}}(\omega_2^{3}\chi)$ is the first bullet point of \cite[Theorem 4.2.4.7(iii)]{BreuilMezard-Multiplicities}, where the reader should take $k = 4 < p$ and $\ell = v_p(\mathcal L) < 0$.

Now we assume that either $h \geq 4$ or $p = h = 3$. Then, applying Proposition \ref{prop:descent-almostready}, Lemma \ref{lemma:Gintegral-properties}, and Proposition \ref{prop:descent-prop}, we deduce that there exists a basis of $\mathcal M_2$ in which the matrix of $\varphi_{\mathcal M_2}$ is given by $A = \begin{smallpmatrix} P & - 1 \\ E^h & 0 \end{smallpmatrix}$ and $P \in \mathfrak m_F[u]$. Define $\mathfrak M = \mathfrak S_{\Lambda}^{\oplus 2}$ with the matrix of $\varphi$ being given by $A$. Clearly $\mathfrak M$ is a Kisin module over $\mathfrak S_{\Lambda}$ of height $\leq h$, and
$$
\mathfrak M \otimes_{\mathfrak S_{\Lambda}} R_2 \cong \mathcal M_2 = \underline{\mathcal M}_{\mathcal O_F}(D_{h+1,\mathcal L})\otimes_{\mathcal O_F} R_2
$$
as $\varphi$-modules over $R_2$. Thus, by Proposition \ref{prop:reductions-first} we deduce $\mathfrak M = \mathfrak M(T)$ for some lattice $T \subseteq V_{h+1,\mathcal L}$. Furthermore, $\mathfrak M \otimes_{\mathfrak S_{\Lambda}} \mathbb F[u^{-1}]$ is a $\varphi$-module over $\mathbb F(\!(u)\!)$ with Frobenius given by $\begin{smallpmatrix} 0 & - 1 \\ u^h & 0 \end{smallpmatrix}$. This shows, in particular, that $\overline V_{h+1,\mathcal L}$ is the same for any $\mathcal L$ satisfying \eqref{eqn:blanket-assumptions} (see \cite[Corollary 2.3.2]{BergdallLevin-BLZ}).

Let $V_{h+1,\infty}$ be as in the introduction. By \cite[Corollary 5.2.2]{BergdallLevin-BLZ}, for  $V_{h+1,\infty}$ there exists a Kisin module $\M'$  such that  $M':= \mathfrak M'  \otimes_{\mathfrak S_{\Lambda}} \mathbb F[u^{-1}]$ has Frobenius also given by $\begin{smallpmatrix} 0 & - 1 \\ u^h & 0 \end{smallpmatrix}$ and $M'$ determines $\overline V_{h+1,\infty} \cong \operatorname{Ind}_{G_{\mathbb Q_p^2}}^{G_{\mathbb Q_p}}(\omega_2^{h}\chi)$. Therefore, $\overline V_{h+1,\mathcal L} \cong \overline V_{h+1,\infty} \cong \operatorname{Ind}_{G_{\mathbb Q_p^2}}^{G_{\mathbb Q_p}}(\omega_2^{h}\chi)$. 
\end{proof}

\begin{rmk}\label{remark:final-remark}
We return to Remark \ref{rmk:small-p}. Suppose we replace \eqref{eqn:blanket-assumptions} with
\begin{equation}\label{eqn:p3-assumptions}
v_p(\mathcal L^{-1}) > {h-1\over 2} + v_p((h-1)!).
\end{equation}
This has the impact of scaling $z$ by a $p$-adic unit multiple of $p$, thus increasing $v_{R_2}(z)$ by $2$ throughout our estimates in Section \ref{subsec:Rm-basis}. The reader may check that Proposition \ref{prop:descent-almostready} holds with these new estimates, and so the proof goes through for all $h \geq 2$ and $p \geq 3$ under the assumption \eqref{eqn:p3-assumptions}. Of course, this bound is not the sharpest possible when $h=2$ or $h=3$. For instance, we've already noted that for $h=3$ and $p \geq 5$, Breuil and M\'ezard confirmed Theorem \ref{thm:main-thm-in-text} with the stronger bound \eqref{eqn:blanket-assumptions}.

The situation is more complicated when $h = 2$. In that case, for $p \geq 5$, Guerberoff and Park showed that $\overline V_{3,\mathcal L} \cong \operatorname{Ind}_{G_{\mathbb Q_p^2}}^{G_{\mathbb Q_p}}(\omega_2^{2}\chi)$  exactly on $v_p(\mathcal L - 1) < \frac{1}{2}$ (see \cite[Theorem 5.0.5]{GuerberoffPark-Semistable}). Thus, the bound $v_p(\mathcal L) < \frac{1}{2}$ from Theorem \ref{thm:main-thm-in-text} produces too large a region of $\mathcal L$-invariants, whereas \eqref{eqn:p3-assumptions} produces a region too small. For the interested reader, Guerberoff and Park also determined, for any $\mathcal L$, the restriction of $\overline V_{3,\mathcal L}$ to the inertia subgroup. The restriction to inertia was recently removed by Chitrao, Ghate, and Yasuda using a completely different method. See \cite[Theorem 1.3]{ChitraoGhateYasuda-Semistable}. Thus we have a complete picture of $\overline V_{3,\mathcal L}$. It would be amusing to understand if that picture can be recovered from the method here.
\end{rmk}

%% file: main_v_final.bbl
\begin{thebibliography}{10}

\bibitem{ABGT-Linvariants}
S.~Anni, G.~B\"ockle, P.~Gr\"af, and A.~Troya.
\newblock Computing $\protect \mathcal{L}$-invariants via the
  greenberg--stevens formula.
\newblock {\em Journal de Th\'eorie des Nombres de Bordeaux}, 31(3):727--746,
  2019.

\bibitem{Bergdall-ConstantSlopes}
J.~Bergdall.
\newblock Upper bounds for constant slope {$p$}-adic families of modular forms.
\newblock {\em Selecta Math. (N.S.)}, 25(4):Art. 59, 24, 2019.

\bibitem{BergdallLevin-BLZ}
J.~Bergdall and B.~Levin.
\newblock Reductions of some two-dimensional crystalline representations via
  {K}isin modules.
\newblock {\em To appear in {I}nt.\ Math.\ Res.\ Not.}

\bibitem{Berger-LocalConstancy}
L.~Berger.
\newblock Local constancy for the reduction mod {$p$} of 2-dimensional
  crystalline representations.
\newblock {\em Bull. Lond. Math. Soc.}, 44(3):451--459, 2012.

\bibitem{BergerLiZhu-SmallSlopes}
L.~Berger, H.~Li, and H.~J. Zhu.
\newblock Construction of some families of 2-dimensional crystalline
  representations.
\newblock {\em Math. Ann.}, 329(2):365--377, 2004.

\bibitem{Breuil-Griffiths}
C.~Breuil.
\newblock Repr\'esentations {$p$}-adiques semi-stables et transversalit\'e de
  {G}riffiths.
\newblock {\em Math. Ann.}, 307(2):191--224, 1997.

\bibitem{BreuilMezard-Multiplicities}
C.~Breuil and A.~M{\'e}zard.
\newblock Multiplicit\'es modulaires et repr\'esentations de {${\rm GL}\sb
  2({\bf Z}\sb p)$} et de {${\rm Gal}(\overline{\bf Q}\sb p/{\bf Q}\sb p)$} en
  {$l=p$}.
\newblock {\em Duke Math. J.}, 115(2):205--310, 2002.
\newblock With an appendix by Guy Henniart.

\bibitem{BuzzardCalegari-GouveaMazur}
K.~Buzzard and F.~Calegari.
\newblock A counterexample to the {G}ouv\^ea-{M}azur conjecture.
\newblock {\em C. R. Math. Acad. Sci. Paris}, 338(10):751--753, 2004.

\bibitem{BuzzardGee-SmallSlope}
K.~Buzzard and T.~Gee.
\newblock Explicit reduction modulo {$p$} of certain two-dimensional
  crystalline representations.
\newblock {\em Int. Math. Res. Not. IMRN}, (12):2303--2317, 2009.

\bibitem{ChitraoGhateYasuda-Semistable}
A.~Chitrao, E.~Ghate, and S.~Yasuda.
\newblock Semi-stable representations as limits of crystalline representations.
\newblock {\em Preprint}, 2021.
\newblock Available at {\tt arXiv:2109.13676}.

\bibitem{ColemanStevensTeitelbaum-Experiments}
R.~Coleman, G.~Stevens, and J.~Teitelbaum.
\newblock Numerical experiments on families of {$p$}-adic modular forms.
\newblock In {\em Computational perspectives on number theory ({C}hicago, {IL},
  1995)}, volume~7 of {\em AMS/IP Stud. Adv. Math.}, pages 143--158. Amer.
  Math. Soc., Providence, RI, 1998.

\bibitem{Colmez-Trianguline}
P.~Colmez.
\newblock Repr\'esentations triangulines de dimension 2.
\newblock {\em Ast\'erisque}, (319):213--258, 2008.
\newblock Repr{\'e}sentations $p$-adiques de groupes $p$-adiques. I.
  Repr{\'e}sentations galoisiennes et $(\phi,\Gamma)$-modules.

\bibitem{Colmez-L_invariants}
P.~Colmez.
\newblock Invariants {$\mathscr L$} et d\'eriv\'ees de valeurs propres de
  {F}robenius.
\newblock {\em Ast\'erisque}, (331):13--28, 2010.

\bibitem{ColmezFontaine-WAdmImpliesAdm}
P.~Colmez and J.-M. Fontaine.
\newblock Construction des repr\'esentations {$p$}-adiques semi-stables.
\newblock {\em Invent. Math.}, 140(1):1--43, 2000.

\bibitem{Fontaine-RepresentationSemiStable}
J.-M. Fontaine.
\newblock Repr\'esentations {$p$}-adiques semi-stables.
\newblock {\em Ast\'erisque}, (223):113--184, 1994.
\newblock With an appendix by Pierre Colmez, P{\'e}riodes $p$-adiques
  (Bures-sur-Yvette, 1988).

\bibitem{Graef-Linvariant}
P.~M. Gr{\"a}f.
\newblock A control theorem for {$p$}-adic automorphic forms and {T}eitelbaum's
  {$\mathcal {L}$}-invariant.
\newblock {\em Ramanujan J.}, 50(1):13--43, 2019.

\bibitem{GuerberoffPark-Semistable}
L.~Guerberoff and C.~Park.
\newblock Semistable deformation rings in even {H}odge-{T}ate weights.
\newblock {\em Pacific J. Math.}, 298(2):299--374, 2019.

\bibitem{Kisin-FCrystals}
M.~Kisin.
\newblock Crystalline representations and {$F$}-crystals.
\newblock In {\em Algebraic geometry and number theory}, volume 253 of {\em
  Progr. Math.}, pages 459--496. Birkh\"auser Boston, Boston, MA, 2006.

\bibitem{Laffaille}
G.~Laffaille.
\newblock Construction de groupes {$p$}-divisibles. {L}e cas de dimension
  {$1$}.
\newblock In {\em Journ\'{e}es de {G}\'{e}om\'{e}trie {A}lg\'{e}brique de
  {R}ennes. ({R}ennes, 1978), {V}ol. {III}}, volume~65 of {\em Ast\'{e}risque},
  pages 103--123. Soc. Math. France, Paris, 1979.

\bibitem{Liu-Torsion}
T.~Liu.
\newblock Torsion {$p$}-adic {G}alois representations and a conjecture of
  {F}ontaine.
\newblock {\em Ann. Sci. \'Ecole Norm. Sup. (4)}, 40(4):633--674, 2007.

\bibitem{Liu-LatticesInSemiStable}
T.~Liu.
\newblock On lattices in semi-stable representations: a proof of a conjecture
  of {B}reuil.
\newblock {\em Compos. Math.}, 144(1):61--88, 2008.

\bibitem{Liu-3dimensional}
T.~Liu.
\newblock Reductions of certain 3-dimensional crystalline representations.
\newblock {\em Preprint}, 2021.
\newblock Available at {\tt arXiv:2107.11009}.

\bibitem{Mazur-Monodromy}
B.~Mazur.
\newblock On monodromy invariants occurring in global arithmetic, and
  {F}ontaine's theory.
\newblock In {\em {$p$}-adic monodromy and the {B}irch and {S}winnerton-{D}yer
  conjecture ({B}oston, {MA}, 1991)}, volume 165 of {\em Contemp. Math.}, pages
  1--20. Amer. Math. Soc., Providence, RI, 1994.

\bibitem{Saito-padicHodgeTheory}
T.~Saito.
\newblock Modular forms and {$p$}-adic {H}odge theory.
\newblock {\em Invent. Math.}, 129(3):607--620, 1997.

\end{thebibliography}
